\documentclass[12 pt]{amsart}
\usepackage{amsmath, amsthm, amssymb}
 \usepackage{amsfonts}
\usepackage{amsmath}
\usepackage{amsbsy}
\usepackage[all,cmtip]{xy}
\usepackage{amscd}
\usepackage{enumitem}
\usepackage[all]{xy}
\usepackage{graphicx}
\usepackage{ifpdf}
\usepackage[UKenglish]{babel}
\usepackage{hyperref}
\usepackage{color}
\newcommand{\Q}{\mathbb{Q}}
\newcommand{\Z}{\mathbb{Z}}
\newcommand{\F}{\mathbb{F}}

\newcommand{\G}{\mathfrak{G}}
\newcommand{\tc}{\textcolor}

\newcommand{\da}{\ar@{-->}}
\newcommand{\dar}{\ar@{.>}}
\newcommand{\lar}{\ar@{-}}
\newcommand{\del}{\partial}

\newcommand{\mf}{\mathfrak}

\newcommand{\bu}{\bullet}

\newtheorem{theorem}{Theorem}[section]
\newtheorem{lemma}[theorem]{Lemma}
\newtheorem{prop}[theorem]{Proposition}
\newtheorem{corollary}[theorem]{Corollary}
\newtheorem{definition}[theorem]{Definition}
\newtheorem{remark}[theorem]{Remark}
\newtheorem{question}[theorem]{Question}

\begin{document}

\author{David Krcatovich}
\title{The reduced knot Floer complex}
\date{\vspace{-3ex}}
\maketitle

\begin{abstract}
We define a \textquotedblleft reduced" version of the knot Floer complex $CFK^-(K)$, and show that it behaves well under connected sums and retains enough information to compute Heegaard Floer $d$-invariants of manifolds arising as surgeries on the knot $K$. As an application to connected sums, we prove that if a knot in the three-sphere admits an $L$-space surgery, it must be a prime knot. As an application to the computation of $d$-invariants, we show that the Alexander polynomial is a concordance invariant within the class of $L$-space knots, and show the four-genus bound given by the $d$-invariant of +1-surgery is independent of the genus bounds given by the Ozsv\'ath-Szab\'o $\tau$ invariant, the knot signature and the Rasmussen $s$ invariant.
\end{abstract}

\section{Introduction }
\label{intro}
In \cite{OSdisks}, Ozsv\'ath and Szab\'o define the Heegaard Floer homology groups (denoted $HF^{\infty}, HF^+, HF^-$ and $\widehat{HF}$) of a three-manifold, which arise as the homology of chain complexes (denoted $CF^{\infty}, CF^+, CF^-$ and $\widehat{CF}$) associated to a pointed Heegaard diagram. We will assume that the reader is familiar with their definitions. Further, given a knot $K$ in $S^3$, Ozsv\'ath and Szab\'o \cite{OSknot} define a $\Z\oplus \Z$-filtered chain complex $CFK^{\infty}(K)$, discovered independently by Rasmussen \cite{Rasmussenknot}, which is freely generated over the ring $\F[U,U^{-1}]$, where $\F$ is the field with two elements (the complex can be defined over $\Z[U,U^{-1}]$, but in this paper we will always assume coefficients in $\F$). In this paper, our primary interest will be in its subcomplex $CFK^-(K)$. Ignoring one of the $\Z$-filtrations on $CFK^-(K)$ gives the complex $CF^-(S^3)$, whose graded homology is isomorphic to $\F[U]$, supported in grading zero. Indeed, the $\Z\oplus \Z$-filtered complex contains all the information necessary to compute the Heegaard Floer homology of not merely $S^3$, but also manifolds arising from Dehn surgery along $K$ \cite{OSinteger,OSrational}.

Our method here will be to ignore the other $\Z$-filtration -- the one which measures the exponent of the variable $U$ -- and to simplify $CFK^-(K)$ to a $\Z$-filtered chain homotopy equivalent complex, which we will call the {\it{reduced}} $CFK^-(K)$, denoted $\underline{CFK}^-(K)$. Being $\Z$-filtered chain homotopy equivalent to $CFK^-(K)$, this complex will still allow us to compute $HFK^-(K)$, the homology of its associated graded object. But we will further require that the reduced complex keep track of the endomorphism given by multiplication by $U$, in a sense made precise in Section \ref{knotfloer}. 

Recall that in \cite[Theorem 7.1]{OSknot}, it was shown that there is a tensor product formula for the knot Floer complexes of connected sums, \[CFK^-(K_1\# K_2)\cong CFK^-(K_1)\otimes_{\F[U]} CFK^-(K_2).\] These tensor product complexes become difficult to work with, however, even for simple sums. The essence of the following theorem is that, while the reduced complex is smaller, is still has a simple tensor product formula under connected sums.

\begin{theorem}
\label{tensorintro}
If $K_1$ and $K_2$ are knots in $S^3$, then \[\underline{CFK}^-(K_1)\otimes_{\F [U]} CFK^-(K_2)\] is a $(\Z,U)$-filtered chain deformation retract of $CFK^-(K_1\# K_2)$.
\end{theorem}

If we wish to connect sum a third knot, we can now reduce $CFK^-(K_1\# K_2)$ and iterate Theorem \ref{tensorintro}. As a result, this object can greatly simplify computations for sums of knots, and, as we will show, still retains enough information for these computations to be useful.

One application is the following. Recall that in \cite{OSlens}, Ozsv\'ath and Szab\'o define a rational homology three-sphere $Y$ to be an $L$-{\it{space}} if -- like a lens space -- it has the smallest possible Heegaard Floer homology. That is,
\begin{equation}
\widehat{HF}(Y,\mf t) \cong \F \text{ for all spin$^c$ structures } \mf t.
\label{lspacedef}
\end{equation}
A knot $K\subset S^3$ is called an $L$-space knot if $n$-surgery on $K$ is an $L$-space, for some positive integer $n$. Examples include positive torus knots (or any knot with a positive lens space surgery) and the $P(-2,3,2n+1)$ pretzel knots \cite{OSlens}, and more generally, a family of twisted torus knots \cite{Vafaeetwisted}. By combining work of Hedden and Hom \cite{HeddenCableII,HomCable}, the $(p,q)$-cable of a knot $K$ is an $L$-space knot if and only if $K$ is an $L$-space knot and \[\frac{q}{p}\geq 2g(K)-1,\] where $g$ is the Seifert genus. In \cite{Vafaeetwisted}, Vafaee asks if there are any other satellite operations which can produce $L$-space knots. We give a negative answer for the simplest satellite operation, connected sums. After describing the reduced complexes of $L$-space knots, we use Theorem \ref{tensorintro} to prove
\begin{theorem}
\label{primeintro}
A knot in $S^3$ which admits an $L$-space surgery must be a prime knot.
\end{theorem}
\noindent We should remark here that it is easy to see that the sum of two non-trivial $L$-{\it{space knots}} is not an $L$-space knot; for example, by observing that the characterization of knot Floer complexes of $L$-space knots given in \cite[Theorem 1.2]{OSlens} is not preserved under tensor products. However, our reduced complex will make this statement just as apparent for the sum of {\it{any}} non-trivial knots.

Our remaining applications will pertain to the {\it{correction terms}}, or $d$-{\it{invariants}}. In \cite[Definition 4.1]{OSabsgr}, Ozsv\'ath and Szab\'o define the $d$-invariant of a spin$^c$ rational homology three-sphere $(Y,\mf t)$ as
\begin{equation}
d(Y,\mf t) = \min \{ \widetilde{gr}(x) |\ x \text{ in the image of } \pi_*:HF^{\infty}(Y,\mf t)\to HF^+(Y,\mf t) \},
\label{ddef}
\end{equation}
where $\widetilde{gr}$ is the absolute lift of the homological $\Z$-grading to $\Q$. We will work with the equivalent definition
\begin{equation}
d(Y,\mf t)=\max \{\widetilde{gr}(x) |\ x \in HF^-(Y,\mf t), x \text{ is not } U\text{-torsion}. \}
\label{ddef2}
\end{equation}
\begin{remark}
Our convention which makes these definitions agree is slightly different than that of Ozsv\'ath and Szab\'o - we assume {\it{both}} $CF^+$ and $CF^-$ to contain the element 1 in $\F[U,U^{-1}]$. This will be convenient for computing correction terms, but has the drawback that $CF^+$ is not quite the quotient complex corresponding to the subcomplex $CF^-$. 
\label{convention}
\end{remark}
These invariants have been used to answer questions related to Dehn surgery \cite{DoigFinite,DoigObstructing,NiWuCosmetic}, the smooth knot concordance group \cite{ChaRuberman,HLR,JabukaNaik,ManolescuOwens} and various notions of genera of knots \cite{Batson,GilmerLivingston,NiWu}.

The property of \textquotedblleft keeping track of multiplication by $U$" which we ascribed to the reduced complex above is essential, since it will allow us to compute $d$-invariants.

Given a knot $K\subset S^3$, one can consider $S^3_1(K)$, the rational homology sphere arising from Dehn surgery along $K$ with slope 1. This has only one spin$^c$ structure, so we can define
\begin{equation}
d_1(K) = d(S^3_1(K)).
\label{d1def}
\end{equation}
This invariant of $K$ was studied by Peters in \cite{Peters}, where he showed, in particular, that $d_1(K)$ is a concordance invariant of $K$, and that it gives a lower bound on the smooth four-dimensional genus of $K$, \[0\leq -d_1(K) \leq 2g_4(K).\] He also showed how this invariant can be computed from $CFK^{\infty}(K)$. There is, of course, another four-genus bound which is defined in terms of knot Floer homology, namely the Ozsv\'ath-Szab\'o $\tau$ invariant. In comparing the computation of these two invariants, Peters poses the question:

\begin{question}
What is the relation between $d_1$ and $\tau$? Is it necessarily true that \[|d_1(K)| \leq 2|\tau(K)| ?\]
\label{petersquestion}
\end{question}
Of course, if this were true, the four-genus bound provided by $d_1$ would be rather ineffective. After all, $\tau$ can be computed just as easily from the knot Floer complex, and has the further advantage of being additive under connected sums. It is fortunate, in this sense, that we can give a negative answer\footnote{The first knot in a family of  knots described in Section \ref{applications} which provides the negative answer was actually alluded to by Peters later in his paper.} to Question \ref{petersquestion}, and show that, in fact, $d_1$ can be quite useful as a four-genus bound. Denoting the mirror image of a knot $K$ by \textquotedblleft $-K$", we first show that

\begin{theorem}
\label{alexanderconcordanceintro}
Suppose that $K_1$ and $K_2$ are two knots in $S^3$ which admit $L$-space surgeries. If $$d_1(K_1\#-K_2)=d_1(-K_1\#K_2)=0,$$ then $$\Delta_{K_1}(T)=\Delta_{K_2}(T).$$ In particular, the Alexander polynomial is a concordance invariant of $L$-space knots.
\end{theorem}

Following from this, we have the corollary

\begin{corollary}
If $K_1$ and $K_2$ are two $L$-space knots whose Alexander polynomials are distinct but have the same degree, then $$\tau(K_1\# -K_2) = \tau (-K_1\# K_2) =0,$$ but either $$d_1(K_1\# -K_2)\neq 0 \ \ \ \textrm{ or } \ \ \ d_1(-K_1\# K_2)\neq 0.$$ In particular, $d_1$ gives a stronger four-genus bound than $\tau$ for $K_1\# -K_2$ and its mirror.
\end{corollary}

In addition to $\tau$, two other concordance invariants which have proven to give useful four-genus bounds are the knot signature $\sigma$ and the Rasmussen $s$ invariant which comes from Khovanov homology \cite{RasmussenSlice}. To strengthen the result of this Corollary, and show the effectiveness of $d_1$ as a smooth four-genus bound, we give examples of knots for which $\tau(K)=\sigma(K)=s(K)=0$, but $|d_1(K)|$ is arbitrarily large. \\

\noindent{\bf{Organization.}} In Section \ref{prelim}, we begin by introducing the algebraic framework which will be necessary. In Section \ref{knotfloer}, we review the definition and properties of the knot Floer complex, and define its reduced form. Subsection \ref{sums} explains how the tensor product formula extends to the reduced complex. In Section \ref{applications}, we apply the theory to $L$-space knots, prove Theorems \ref{primeintro} and \ref{alexanderconcordanceintro}, and provide examples.\\

\noindent {\bf{Acknowledgements.}} The author would like to thank his advisor, Matthew Hedden, for his insight and patience. Also, Chuck Livingston, Maciej Borodzik and Margaret Doig helped clarify some of these ideas through their discussions, and Faramarz Vafaee provided helpful comments on an earlier draft. The author was partially supported by National Science Foundation RTG Grant DMS 0739208.

\section{Algebraic preliminaries}
\label{prelim}

Throughout this paper, we will be working with coefficients in the field with two elements, which we will denote $\F$. Given a chain complex $(C,\del)$, and a partially ordered set $S$, a {\it{(decreasing) $S$-filtration}} on $C$ is a function $F:C\to S$ such that, for all $x$ and $y$ in $C$, 
\begin{equation}
F(x+y)\leq F(x)  \ \ \ \text{or}\ \ \  F(x+y)\leq F(y),
\label{max}
\end{equation} and $$F (\del x) \leq F(x).$$ To satisfy \eqref{max}, we will further require that $S$ contains an element \textquotedblleft $-\infty$", satisfying \( -\infty \leq x\) for all \(x\in S\), and that \( F^{-1}(-\infty)=0\).

These properties ensure that the sets $$F_i:=\{ x\in C | F(x)\leq i\}$$ are subcomplexes, with $$\cdots F_{i-1}\subseteq F_i\subseteq F_{i+1} \cdots$$  We call $(C,\del,F)$ a {\it{filtered complex}}. To simplify notation, we will omit $\del$ or $F$ when it does not cause confusion to do so. The filtration is said to be {\it{ bounded (above)}} if $F_i=C$ for sufficiently large $i$. 

The complexes dealt with here will be filtered by $\Z$ or by $\Z \oplus \Z$  (each including \(-\infty\)), with partial ordering $$(i,j)\leq (i',j') \ \ \ \text{iff}\ \ \ i\leq i' \ \ \text{and} \ \ j\leq j',$$ and all filtered complexes will be bounded. 

If $(C,\del,F)$ and $(C',\del ',F')$ are filtered complexes (filtered by the same partially ordered set), then a map $f:C\to C'$ is a {\it{filtered map}} if, for all $x\in C$, $$F'(f(x)) \leq F(x).$$

Suppose $(C,\del )$ and $(C',\del ')$ are filtered chain complexes. We will say that $C$ and $C'$ are {\it{ filtered chain homotopy equivalent}} if there exist filtered chain maps $f:C\to C'$ and $g:C'\to C$, and filtered chain homotopies $h:C\to C$ and $h':C'\to C'$ such that $$f\circ g = I_{C'} + \del'h' + h'\del' \ \ \ \text{and}\ \ \ g\circ f =I_C + \del h+h\del .$$  We will further say that $C'$ is a {\it{filtered chain deformation retract of $C$}} if the chain homotopy $h'$ is trivial; i.e., if $$f\circ g=I_{C'}.$$

We present here a prototypical example of what will follow. Figure \ref{example} represents a $\Z$-filtered complex $C$ generated over $\F$, where the vertical height of each generator corresponds to its filtration level. We will denote by \( \del (x,y)\) the coefficient of $y$ in $\del x$. If, for example, \( \del (a,c)=1\), we draw an arrow from $a$ to $c$. Intuitively, we can \textquotedblleft cancel" an arrow which is horizontal, while preserving the filtered chain homotopy type of $C$. For example, canceling the arrow from $b$ to $c$  gives a complex $C'$ in the following way. The generators are obtained by deleting the generators $b$ and $c$, and the differential on $C'$ is given by \[ \del' (x,y)=\del(x,y) +\del(x,c)\del(b,y).\] In other words, if an arrow went from $x$ to $c$, and another arrow went from $b$ to $y$, we add an arrow going from $x$ to $y$. To put it precisely, we make the filtered change of basis $c \mapsto \del b$, then take the quotient of $C$ by the acyclic subcomplex which is generated by $b$ and $\del b$. If we define a homomorphism $h:C\to C$ by setting $h(c)=b$ and $h(x)=0$ for all other generators (i.e., $h$ is the inverse of the horizontal arrow we are canceling), then $C'$ is seen to be a filtered chain deformation retract of $C$, via the maps 
\[ f=\pi \circ  (I+\del h), \ \ \ g=(I+h\del)\circ \iota \]
and the chain homotopy $h$. Details are explained well in \cite[Lemma 4.1]{HeddenNi} and \cite[Section 5.1]{Rasmussenknot}, and will be worked out in Section \ref{knotfloer}.

\begin{figure}
$$\xymatrix{
&&& \underline{F} &&&&& \underline{F'} \\
&& a \ar[dll] \ar[dl] & 1 &&&& a \ar[dl] \ar[dd] & 1 \\
d & c & b \ar[l] \ar[d] & 0 & \ar@{~>}[r] & & d && 0 \\
&& e & -1 &&&& e & -1\\
& C &&&&&& C'
}$$

\caption{Canceling the horizontal arrow from $b$ to $c$ yields the filtered chain deformation retract $C'$. The additional arrow from $a$ to $e$ is obtained by \textquotedblleft traveling backward" through the canceled arrow.}
\label{example}
\end{figure}
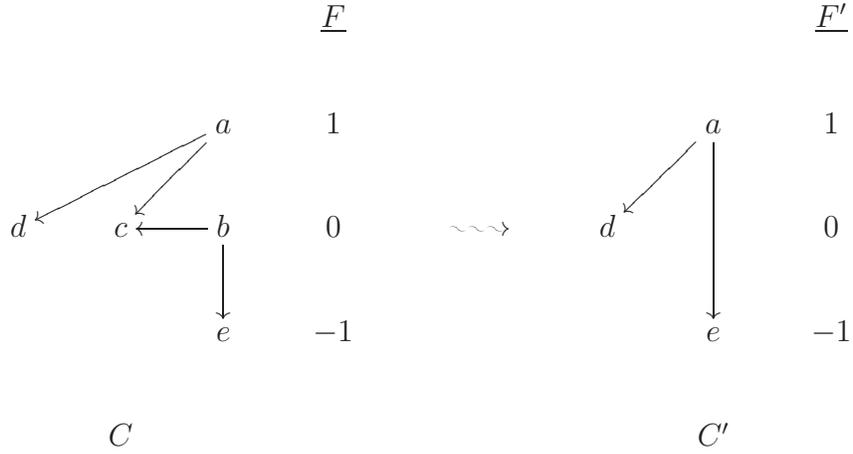

 A vertical arrow could similarly be canceled, but the map $h$ would not be filtered, and so the result would be a chain homotopy equivalent, but not filtered chain homotopy equivalent, complex. The idea of reduction, put simply, is that as long as we have horizontal arrows (terms in the differential which preserve the filtration level), we can reduce the number of generators of a chain complex, while maintaining its filtered chain homotopy type.

Suppose $C$ is a $\Z$-filtered complex which comes equipped with a specified filtered chain map $U$. Then we will call the pair $(C,U)$ a $(\Z,U)$-{\it{filtered chain complex}}. It will be convenient to view such complexes as modules over $\F[U]$, so that we will often refer to composition with the map $U$ as \textquotedblleft multiplication by $U$". Two $(\Z,U)$-filtered chain complexes $(C,U)$ and $(C',U')$ will be called $(\Z,U)$-{\it{filtered chain homotopy equivalent}} if they are filtered chain homotopy equivalent, and the maps $f$ and $g$ respect multiplication by $U$, in the sense that \[ gU'\sim Ug \ \ \ \ \text{and} \ \ \ \ \ fU\sim U'f,\]
where $\sim$ signifies that the maps are filtered chain homotopic.

\begin{remark}
\label{retractcomp}
The notion of $(\Z,U)$-filtered chain homotopy equivalence is an equivalence relation.
\end{remark}

Finally, we will also want to consider complexes which arise as products of other complexes. If $(C_1,F_1)$ and $(C_2,F_2)$ are two $\Z$-filtered chain complexes which are freely generated over $\F$, we can define a filtration $F^{\times}$ on the direct product $C_1\times C_2$ in the following way. Suppose that $\{x_i\}$ and $\{y_i\}$ are generating sets for $C_1$ and $C_2$, respectively, so that any element in $C_1\times C_2$ can be written as $\sum_{i,j} \varepsilon_{ij}x_iy_j$, where $\varepsilon_{ij}\in \F$, and all but finitely many of the $\epsilon_{ij}$ are zero. Then,
\[F^{\times}\Bigg(\sum_{i,j} \varepsilon_{ij}x_iy_j \Bigg) = \max \{ F_1(x_i)+F_2(y_j) | \varepsilon_{ij}=1 \}\]
defines a filtration on $C_1\times C_2$.

If each $C_i$ comes with a filtered chain map $U_i$, they can naturally be thought of as modules over $\F[U]$, so that we will refer to the maps $U_i$ as \textquotedblleft multiplication by $U$". In this case, we can also consider the complex \[C_1\otimes_{\F[U]} C_2. \] An element in $C_1\otimes _{\F[U]} C_2$ is an equivalence class of elements in $C_1\times C_2$, and we define a filtration $F$ on the tensor product by simply taking the minimum over each equivalence class. That is, for any $s\in C_1\times C_2$, 
\begin{equation}
F([s]) = \min \left\{ F^{\times}(t) \big| t\in [s] \right\}.
\label{tensorfiltration1}
\end{equation}
Actually, for the tensor products we will consider in this paper, we can describe the product filtration more concretely. In our case, we will consider a $(\Z,U)$-filtered complex $C_1$ for which $U$ is not necessarily homogeneous, but always decreases the filtration level by at least 1; that is, 
\begin{equation}
F_1(Ux)\leq F_1(x)-1 \ \ \text{for all } x \in C_1.
\label{filtge1}
\end{equation}
We will then consider a second $(\Z,U)$-filtered complex, $C_2$, on which $U$ is homogeneous of degree 1, so \[F_2(Ux)=F_2(x)-1 \ \ \text{for all } x,\] and further, $C_2$ is free when viewed as an $\F[U]$-module. In this case, for a homogeneous element $x\in C_1$ and a generator $y\in C_2$, the filtration on $C_1\otimes_{\F[U]} C_2$ given by \eqref{tensorfiltration1} is
\begin{equation}
\label{tensorfiltration}
F(x\otimes U^ny) = F_1(U^nx)+F_2(y),
\end{equation}
for all $n\geq 0$. In other words, to avoid ambiguity, we can think of $U$ as always being applied to the first component (the module on which $U$ is not necessarily homogeneous). In this case, we can prove the following lemma, which we will use in Subsection \ref{sums}.

\begin{lemma}
\label{retracttensor}
Suppose that there is a $(\Z,U)$-filtered chain homotopy equivalence between $C_1$ and $C_1'$, and $C_2$ is a $(\Z,U)$-filtered chain complex which is freely generated over $\F [U]$. Suppose also that the maps $U$ on $C_1$ and $C_1'$ always decrease the filtration level by at least $k$, and the map $U$ on $C_2$ is homogeneous of degree $k$. Then there is a $(\Z,U)$-filtered chain homotopy equivalence between $C_1\otimes_{\F[U]} C_2$ and $C_1'\otimes_{\F[U]} C_2$.
\end{lemma}

\begin{proof}[\sc Proof]
The idea is that, since $C_2$ is freely generated, we can define a map on $C_1\otimes C_2$, for example, by extending a map defined on $C_1$. Further, because the map $U$ on $C_1$ decreases the filtration level by at least as much as the map $U$ on $C_2$, the extended map will still be filtered. 

More precisely, let $f:C_1 \to C_1'$ and $g:C_1' \to C_1$ be the chain maps which give the equivalence, and let $h$ and $h'$ be the chain homotopies from $g\circ f$ to $I_{C_1}$ and from $f\circ g$ to $I_{C_1'}$, respectively. We define chain maps $f:C_1 \otimes C_2 \to C_1'\otimes C_2$ and $g:C_1' \otimes C_2 \to C_1 \otimes C_2$ as follows. Suppose first that $y$ is a generator of $C_2$ as an $\F[U]$-module, then we set
$$f(x\otimes U^ny) = f(U^nx)\otimes y \ \ \ \textrm{and} \ \ \ g(x\otimes U^ny) = g(U^nx)\otimes y,$$ and extend the maps bilinearly over $\F$. Similarly, for $y$ a generator of $C_2$, we define maps $h:C_1\otimes C_2 \to C_1\otimes C_2$  and $h':C_1'\otimes C_2 \to C_1'\otimes C_2$ by setting
\[ h(x\otimes U^ny) = h(U^nx)\otimes y \ \ \ \textrm{and} \ \ \ h'(x\otimes U^ny) = h'(U^nx)\otimes y, \]
and extending bilinearly.

Now we have
\begin{align*}
g\circ f (x\otimes U^ny)=&\ g\circ f (U^nx)\otimes y \\
=&\ (I+\del h + h \del)(U^nx)\otimes y \\
=&\ U^nx\otimes y +\del h(U^nx)\otimes y +h(\del U^nx)\otimes y \\
=&\ U^nx\otimes y + \del h(U^nx)\otimes y+ h(\del U^nx)\otimes y \\
& \ \ \ \ \ \ \ + [h(U^nx)\otimes \del y + h(U^nx)\otimes \del y]\\
=&\ U^nx\otimes y + (\del h(U^nx)\otimes y + h(U^nx)\otimes \del y)\\
&\ \ \ \ \ \ \  + (h(\del U^nx)\otimes y +h(U^nx)\otimes \del y)\\
=&\ (I+\del h +h\del) (x\otimes U^ny),
\end{align*}
and, since the maps are bilinear, we see that $g\circ f$ is chain homotopic to $I_{C_1\otimes C_2}$ via the chain homotopy $h$. By a symmetric argument, $f\circ g \sim I_{C_1'\otimes C_2}$ via $h'$.

Note also that, for example, 
\begin{align*}
F_{C_1'\otimes C_2}(f(x\otimes U^ny))=&\ F_{C_1'\otimes C_2}(f(U^nx)\otimes y)\\
=&\ F_{C_1'}(f(U^nx))+F_{C_2}(y)\\
\leq&\ F_{C_1}(U^nx)+F_{C_2}(y)\\
=&\ F_{C_1\otimes C_2}(x\otimes U^ny).
\end{align*}
In other words, because the maps $f, g, h$ and $h'$ are filtered on $C_1$ and $C_1'$, the maps $f,g,h$ and $h'$ are filtered on $C_1\otimes C_2$ and $C_1'\otimes C_2$. So, the two tensor product complexes are $\Z$-filtered chain homotopy equivalent.

Finally, we verify that this equivalence is in fact $(\Z,U)$-filtered. By assumption, $fU\sim Uf$ as maps from $C_1$ to $C_1'$, so there exists a map $\phi:C_1\to C_1'$ such that \[fU=Uf+\del_{C_1'}\phi + \phi\del_{C_1}.\] In a similar fashion as before, we define \[\phi:C_1\otimes C_2\to C_1'\otimes C_2 \] by setting \[\phi(x\otimes U^ny) = \phi(U^nx)\otimes y\] when $y$ is a generator of $C_2$, and extending bilinearly over $\F$. 

We then check that
\begin{align*}
f(U(x\otimes U^ny)) =&\ f(U^{n+1}x)\otimes y\\
=&\ (Uf+\del\phi + \phi\del)(U^nx)\otimes y \\
=&\ (Uf+\del\phi + \phi\del)(U^nx)\otimes y \\
&\ \ \ \ \ \  + [\phi(U^nx)\otimes \del y + \phi(U^nx)\otimes \del y]\\
=&\ (Uf + \del\phi +\phi\del)(x\otimes U^ny).
\end{align*}
That is, $fU\sim Uf$, and by the same reasoning, $Ug\sim gU$.
\end{proof}
This lemma will be the key to simplifying computations for connected sums of knots, for the complexes which we will define below.

\section{Reducing the knot Floer complex}
\label{knotfloer}

We now turn to the complexes of interest in this paper. Given a pointed Heegaard diagram for a three-manifold $Y$ (with $k$ $\alpha$-curves and $k$ $\beta$-curves), Ozsv\'ath and Szab\'o \cite{OSdisks} define a $\Z$-filtered chain complex $CF^{\infty}(Y,\mf t)$ for each Spin$^c$ structure $\mf t$ on $Y$. The complex is freely generated over \( \F[U,U^{-1}]\) by $k$-tuples of intersection points on the Heegaard diagram. We will denote the set of generators by $\mathfrak{G}$. The filtration level of the homogeneous element $U^nx$ is $-n$, for any $x\in \G$ and any $n\in \Z$. Adding a second basepoint to the Heegaard diagram specifies a knot $K$ in $Y$. Using this additional basepoint (and fixing a Seifert surface for $K$), each $x_i\in \G$ can be assigned an integer $A(x_i)$, called the {\it{Alexander grading}}. If we let \(y_i\) denote a homogeneous element, then after setting 
\begin{equation}
A(U^nx_i)=A(x_i)-n, \ \ \ A\left( \sum_i y_{i} \right) = \max_i \{ A(y_{i}) \},
\label{maxalex}
\end{equation}
$A$ defines an additional filtration on $CF^{\infty}(Y,\mf t),$ discovered in \cite{OSknot}, and independently by Rasmussen \cite{Rasmussenknot}. The $\Z \oplus \Z$-filtered chain homotopy type of this complex is an invariant of $K\subset Y$, and in the case $Y\cong S^3$, this invariant is denoted $CFK^{\infty}(K)$. 

We can write the \( \Z \oplus \Z\)-filtration level of the homogeneous element $U^nx$ as \[ F(U^nx)=(-n,A(x)-n)\] for any $x\in \G$ and $n\in \Z$. It will be convenient to represent these complexes graphically in the $(i,j)$-plane, where $U^nx$ will be represented by a dot with coordinates $(-n, A(x)-n)$. If $x$ and $y$ are two homogeneous elements such that \( \del (x,y)=1\), then we will draw an arrow from the dot representing $x$ to the dot representing $y$. We should also point out here that $CFK^{\infty}(K)$ comes with a homological $\Z$-grading $M$, called the {\it{Maslov grading}}, and that multiplication by $U$ decreases $M$ by 2; i.e., 
\begin{equation}
\label{maslov}
M(Ux)=M(x)-2.
\end{equation}
The difference between the Maslov gradings of two generators can be read from the Heegaard diagram, and to fix an absolute Maslov grading, we declare that the element 1 in \[H_*(CFK^{\infty}(K)) \cong HF^{\infty}(S^3) \cong \F[U,U^{-1}] \] has Maslov grading zero.

Following convention, for a subset $S\subset \Z\oplus \Z$, we will denote by $C\{S\}$ the elements of $C$ whose $(i,j)$-coordinates are contained in $S$, along with the arrows between these elements. We will often consider the subcomplex $C\{i\leq 0\}$, which is written as $CFK^-(K)$. Figure \ref{trefoil} shows the complex $CFK^-(K)$ in the case where $K$ is the right-handed trefoil.

\begin{figure}
$$\xymatrixcolsep{0.5 pc}\xymatrixrowsep{1 pc}
\xymatrix{
& & & & \bu \ x_{(0)} & 1 \\
& & & \bu & \bu \ y_{(-1)} \ar[l] \ar@<-3ex>[d] & j=0 \\
& &  \bu & \bu \ar[l] \ar[d] &  \bu \ z_{(-2)} & -1 \\
&  \bu & \bu \ar[l] \ar[d] & \bu & & -2 \\
& \da[ld] & \bu & &  &-3 \\
\\
& -3 & -2 & -1 & i=0  & 
 }$$
\caption{The complex $CFK^-(T(2,3))$. The generators are represented by dots in the $i=0$ column with their Maslov grading in parentheses, and multiplication by $U$ translates an element one row down and one column to the left. The arrows represent the nonzero terms of the differential.}
\label{trefoil}
\end{figure}
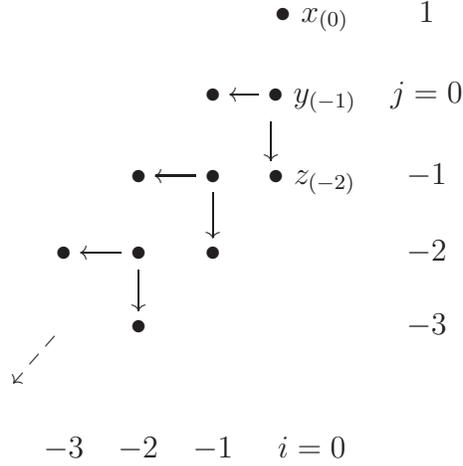

Saying that $A$ defines a filtration means in particular that $A(\del x)\leq A(x)$ for any $x$, and we will call the part of the differential which preserves the Alexander grading the {\it{horizontal differential}}, denoted $\del_H$. Diagramatically, $\del_H$ simply consists of those arrows which are horizontal.  For example, in Figure \ref{trefoil}, we have $\del y=Ux+z$, and $\del_H y=Ux$. If we restrict our attention to a single Alexander grading -- that is, a single row -- of $CFK^{\infty}$, we get a filtered chain complex which has homology isomorphic to $\F$. The \textquotedblleft simplest" such filtered complex would be one which has generators paired into acyclic summands, and a single isolated generator of homology, which has no arrow going into or out of it. It was shown in \cite[Proposition 11.52]{LOT} that one can always find such a basis for $CFK^{\infty}(K)$, called a {\it{horizontally simplified}} basis. In other words, we can choose a generating set \( \G = \{x_i, y_i, z, |  1\leq i\leq N \},\) such that
\begin{itemize}
\item $\del_H (y_i) = U^{r_i}x_i$, for some $r_i >0$, \\
\item $\del_H (x_i) = \del_H (z) = 0$.
\end{itemize}
With respect to this basis, the homology of each subquotient complex \( C\{ j=k\} \) (each row) is generated by the class $U^lz$ (where $l$ differs from $k$ by a constant). All other generators are paired by horizontal arrows, and can therefore be canceled as in Section \ref{prelim}. 

We refer back to Figure \ref{trefoil} to make one more observation. The basis $\{x,y,z\}$ shown there is horizontally simplified, but because we are considering the subcomplex $CFK^-(K)$, rather than all of $CFK^{\infty}(K)$, it is not only $U^kz$ which generates homology: $x$ is also homologically nontrivial (the horizontal arrow which would cancel it has been \textquotedblleft cut off"). Starting with any nontrivial knot $K$, if we cancel horizontal arrows in $CFK^-(K)$, we are left with some elements which are eventually canceled for high enough powers of $U$, but not for lower powers of $U$. 

We are now ready to construct the object of interest in this section, the reduced knot Floer complex. It will be convenient to think of $CFK^-(K)$ as being $(\Z,U)$-filtered, by the Alexander grading $A$, rather than $\Z\oplus \Z$-filtered. So, when we refer to the filtration level of an element, we will mean its \( j\)-coordinate in the diagram (although, for aesthetic reasons, we will maintain the appearance of different $i$-coordinates). In this case, the map $U$ is a filtered chain map of degree 1; i.e., $A(Ux)=A(x)-1$. 

Let $K$ be a knot in $S^3$. To simplify notation, let us define $C_0:=CFK^-(K)$. After choosing a horizontally simplified basis 
\[ \{x_i, y_i, z | 1\leq i\leq N\},\] we will reduce the complex by \textquotedblleft canceling" all of the horizontal arrows, as in Section \ref{prelim}. More precisely, let $h_1$ be the $\F$-linear map on $C_0$ which inverts the horizontal differential going from $y_1$ to $x_1$, and all of its $U$-translates,  so that $h_1(U^{r_1+n}x_1)=U^ny_1$ for all $n\geq0$, and $h_1$ is zero on all other homogeneous elements. 

We now define a filtered chain complex $C_1$ which is freely generated over $\F$ as follows. The generators for $C_1$ over $\F$ are obtained from the homogeneous elements of $C_0$ by removing $U^{r_1+n}x_1$ and $U^ny_1$ for all $n\geq 0$. Define maps $f_1:C_0\to C_1$ and $g_1: C_1 \to C_0$ by

\[ f_1 = \pi \circ (I+\del  h_1), \ \ \ g_1 = (I+h_1 \del)\circ \iota,\] where $\pi$ and $\iota$ are the natural projection and inclusion maps. Further, we define a differential and $U$ map on $C_1$:
\begin{equation}
\label{U_i}
\del_{1} := \pi \circ (\del+\del h_1 \del )\circ \iota, \ \ \ \ \ U_{1}:= \pi\circ(U+\del h_1 U)\circ \iota .
\end{equation}

The filtration $A_1$ on $C_1$ is induced by inclusion, $A_1(x) :=A(\iota x)$.
Since the maps $\del, U$ and $h_1$ are all filtered, so are the maps $f_1,g_1,\del_{1},$ and $U_{1}$ defined above. Let us consider the map $U_1$ in more detail. Recall that on the complex $C_0$, $U$ is a homogeneous map of degree 1. The map $U_1$, however, will not be homogeneous. There are two cases to consider. For the generator $U^{r_1-1}x_1$ (the highest remaining $U$-power which was not canceled by a horizontal arrow), 
\begin{align}
\begin{split}
U_1(U^{r_1-1}x_1)=&\ \left( \pi\circ (U+\del h_1U)\circ\iota\right) (U^{r_1-1}x_1) \\
=&\ \pi\circ \big(U^{r_1}x_1 + \del h_1(U^{r_1}x_1)\big)\\
=&\ \pi\circ (U^{r_1}x_1 + \del y_1).
\label{U_1}
\end{split}
\end{align}
Since we chose a basis which is horizontally simplified, $U^{r_1}x_1$ is the only term in $\del y_1$ which has filtration level equal to that of $y_1$. That is, \[ A( U^{r_1}x_1+\del y_1) < A(y_1).\] It follows from \eqref{U_1} then, that 
\begin{equation}
A_1\big(U_1(U^{r_1-1}x_1)\big)< A(y_1) = A_1(U^{r_1-1}x_1)-1,
\label{bentU}
\end{equation}
so the map $U_1$ decreases the filtration level by more than 1. On all other generators of $C_1$, however, the map $U_1$ is equal to $\pi\circ U \circ\iota$, and therefore still decreases the filtration level by exactly 1. Therefore, the complex $C_1$ is of the type mentioned in Equation \eqref{filtge1}; it is $(\Z,U)$-filtered, and $U$ is a filtered map which decreases the filtration level by {\it at least} one.

\begin{remark}
\label{u1deg}
By its definition, the map $h_1$ increases the Maslov grading by 1, and as a result, the map $U_1$ still lowers the Maslov grading by exactly 2.
\end{remark}

\begin{lemma}
\label{che}
$C_1$ is a $(\Z,U)$-filtered chain deformation retract of $C_0$.
\end{lemma}

\begin{proof}[\sc Proof]
We first verify the chain homotopy equivalence. For any homogeneous element $x$ for which \( \pi(x)\neq 0\), 
\begin{align*}
f_1\circ g_1(\pi(x)) = \ & \pi\circ (I+\del h_1)(I+h_1\del )\circ \iota (\pi(x)) \\
= \ & \pi\circ (I+ \del h_1 + h_1\del )(x) \\
= \ & \pi \circ (x + \del(0) +h_1(\del x)) \\
= \ & \pi(x),
\end{align*}
since the image of $h_1$ projects to zero. So, $f_1\circ g_1 = I_{C_1}$.

Next, we consider the composition $g_1\circ f_1$, in three distinct cases. First, if $x$ is any homogeneous element for which \( \pi(x)\neq 0\), then \( h_1(x)=0\), so
\begin{align*}
g_1\circ f_1 (x) = \ & \left( (I+h_1\del )\circ\iota\right) \circ\left( \pi\circ (I + \del h_1)\right) (x) \\
= \ & (I+h_1\del ) (x) \\
= \ & (I+h_1\del + \del h_1) (x).
\end{align*}
Second, for $n\geq 0$, we have
\begin{align}
\begin{split}
g_1\circ f_1 (U^{r_1+n}x_1) = \ &  \left( (I+h_1\del )\circ\iota\right) \circ \left( \pi\circ (I + \del h_1)\right) (U^{r_1+n}x_1) \\
= \ & (I+h_1\del )\circ\iota\circ \pi\circ (U^{r_1+n}x_1 + \del (U^n y_1)) .
\end{split}
\label{g1f1}
\end{align}

Recall that multiplication by $U$ lowers the Maslov grading by 2. Since $ \del (U^ny_1,U^{r_1+n}x_1) =1$, it follows that $\del (U^ny_1,U^{k}x_1) =0$ for any $k\neq r_1+n$. So, the expression $ (U^{r_1+n}x_1 + \del (U^n y_1))$ in \eqref{g1f1} has no terms of the form $U^kx_1$. Similarly, by considering Maslov gradings, it also contains no elements of the form $U^ky_1$. Because of this, the composition $\iota\circ \pi$ in \eqref{g1f1} is the identity, so we again get
$$ g_1\circ f_1 (U^{r_1+n}x_1) = (I+h_1\del + \del h_1) (U^{r_1+n}x_1).$$
 Finally, $g_1\circ f_1 (U^ny_1) = 0$ for all $n\geq 0$ (since $f_1(U^ny_1)=0$). Also, 
\begin{equation*}
(I + h_1\del + \del h_1)(U^ny_1) = U^ny_1 + U^ny_1 + 0 = 0.
\end{equation*} So, we have verified that in all cases, $$g_1\circ f_1 = I+h_1\del + \del h_1,$$ which is to say, $g_1\circ f_1$ is chain homotopic to the identity. 

This shows that $C_1$ is a $\Z$-filtered chain deformation retract of $C_0$. It remains to check that this equivalence respects multiplication by $U$. In most cases, the fact that $f_1$ commutes with $U$ is immediate, because, in most cases, all maps in the definition of $f_1$ commute with $U$. In fact, this is true for every homogeneous element except $U^{r_1-1}x_1$, on which $h_1$ and $U$ do not commute. We verify the claim directly in this case,
\begin{align*}
f_1 U (U^{r_1-1}x_1) = \ & \pi\circ (U+\del h_1 U)(U^{r_1-1}x_1) \\
= \ &  \pi\circ (U+\del h_1 U)\circ\iota\circ \pi\circ (U^{r_1-1}x_1) \\
= \ & U_1 \circ\pi\circ (U^{r_1-1}x_1)\\
= \ & U_1 \circ\pi\circ (I+\del h_1)(U^{r_1-1}x_1)\\
= \ & U_1 f_1 (U^{r_1-1}x_1).
\end{align*}
Similarly, one can verify that $g_1U_1$ is chain homotopic to $Ug_1$ via the (filtered) chain homotopy $h_1Ug_1$.
\end{proof}

Beginning with $C_0$, we have now reduced the number of horizontal arrows and obtained a $(\Z,U)$-filtered chain homotopy equivalent complex $C_1$. The differential $\del_1$ is nearly just $\pi\circ\del$. The exception being that, if there was an arrow going from a homogeneous element to $U^{r_1+n}x_1$, the differential $\del_1$ adds an arrow from that element to the remaining image of $\del (U^ny_1)$ (see the discussion in Section \ref{prelim} and Figure \ref{example}). Note though, that these additional arrows must always decrease the filtration level (in fact, by more than 1). In particular, the basis given for $C_1$ is still horizontally simplified, with the horizontal differential being $\pi\circ \del_H$. This means we can iterate the above process, at each step moving from $C_i$ to $C_{i+1}$ by canceling the horizontal arrows $U^ny_i \to U^{r_i+n}x_i$, and obtaining a $(\Z,U)$-filtered chain homotopy equivalent complex with a filtered chain map $U_i$. If we begin with a basis $\G$ for $C_0$ consisting of $2N+1$ elements, then $C_N$ will have no horizontal arrows, and will be said to be a \textquotedblleft reduced" version of $CFK^-(K)$. An example of this process of reduction is shown for the $(2,7)$-torus knot in Figure \ref{t27}.

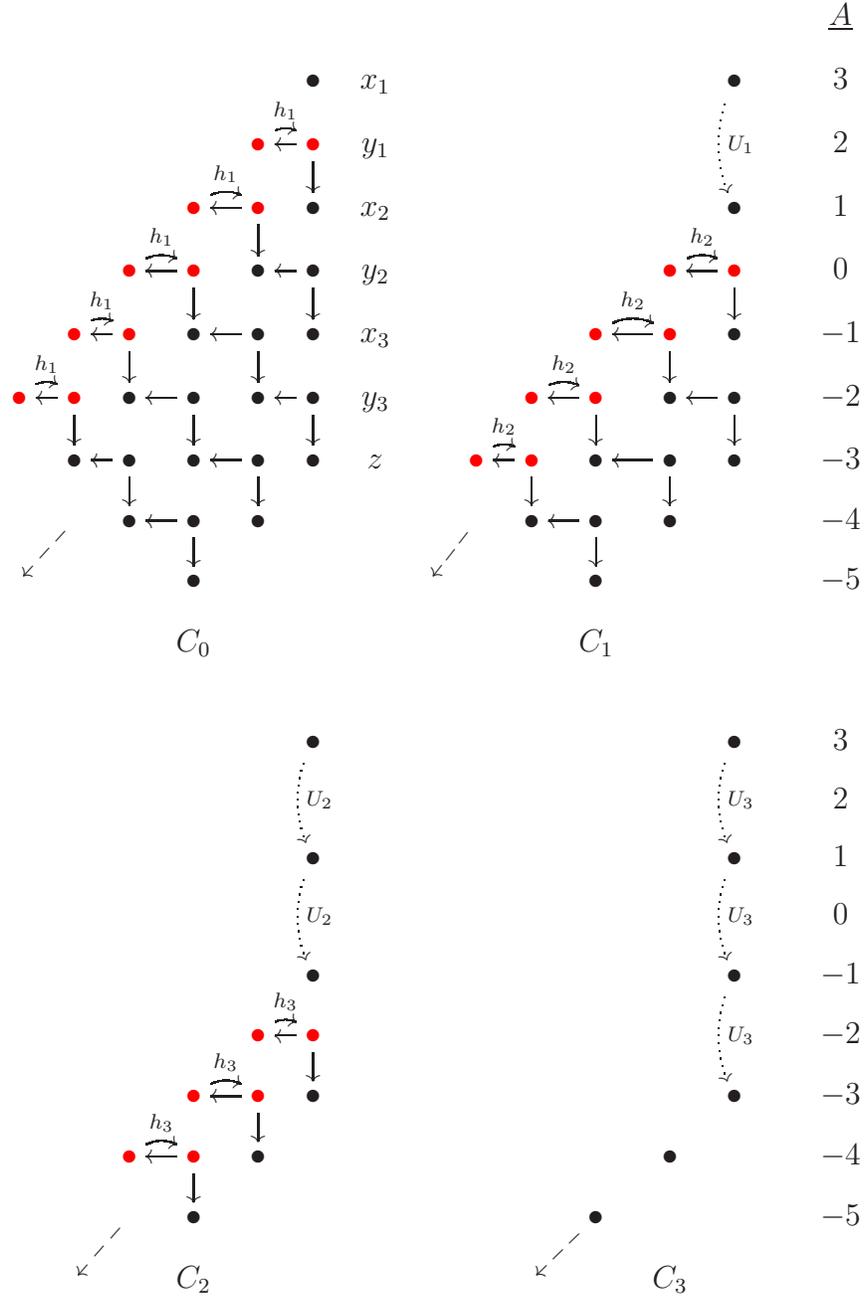
\begin{figure}
$$\xymatrixcolsep{0.65 pc}\xymatrixrowsep{0.6 pc}
\xymatrix{
& & & & & & & & & & & & & & \underline{A} \\
& & & & & \bu & x_1 & &       & & & & \bu \dar@/_/[dd]^{U_1} &                                                                                                     &     3 \\
& & & & \tc{red}{\bu} \ar@/^/[r]^{h_1} & \tc{red}{\bu} \ar[l] \ar[d] & y_1 &   &               & & & & &                                             &    2 \\
& & & \tc{red}{\bu}  \ar@/^/[r]^{h_1} & \tc{red}{\bu} \ar[l] \ar[d] & \bu & x_2 &    &      & & & & \bu &                                       &     1 \\
& & \tc{red}{\bu}  \ar@/^/[r]^{h_1} & \tc{red}{\bu} \ar[l] \ar[d] & \bu & \bu \ar[l] \ar[d] & y_2 &         &       & & & \tc{red}{\bu} \ar@/^/[r]^{h_2} & \tc{red}{\bu} \ar[l] \ar[d] &   &   0 \\
& \tc{red}{\bu}  \ar@/^/[r]^{h_1} & \tc{red}{\bu} \ar[l] \ar[d] & \bu &\bu \ar[l] \ar[d] & \bu & x_3 &  &    & & \tc{red}{\bu} \ar@/^/[r]^{h_2} & \tc{red}{\bu} \ar[l] \ar[d] & \bu &    &  -1 \\
\tc{red}{\bu}  \ar@/^/[r]^{h_1} & \tc{red}{\bu} \ar[l] \ar[d] & \bu & \bu \ar[l] \ar[d] & \bu  & \bu \ar[l] \ar[d] & y_3 & &   & \tc{red}{\bu} \ar@/^/[r]^{h_2} & \tc{red}{\bu} \ar[l] \ar[d] & \bu & \bu \ar[l] \ar[d] &            &         -2 \\
& \bu & \bu \ar[l] \ar[d] & \bu &\bu \ar[l] \ar[d] & \bu & z &   &      \tc{red}{\bu} \ar@/^/[r]^{h_2} & \tc{red}{\bu} \ar[l] \ar[d] & \bu &\bu \ar[l] \ar[d] & \bu &                                                      &  -3 \\
& \da[ld] & \bu & \bu \ar[l] \ar[d] & \bu & &  &                   &       \da[ld] & \bu & \bu \ar[l] \ar[d] & \bu & &                                    &  -4 \\
& & & \bu & & & &                                                     &                & & \bu & & &                                                            &  -5 \\
& & & C_0 & & & & & & & C_1 \\
\\
& & & & & \bu \dar@/_/[dd]^{U_2} & & & & & & & \bu \dar@/_/[dd]^{U_3} & & 3 \\
& & & & & & & & & & & & & & 2 \\
& & & & & \bu \dar@/_/[dd]^{U_2} & & & & & & & \bu \dar@/_/[dd]^{U_3} & & 1 \\
& & & & & & & & & & & & & & 0 \\
& & & & & \bu & & & & & & & \bu \dar@/_/[dd]^{U_3} & & -1 \\
& & & & \tc{red}{\bu} \ar@/^/[r]^{h_3} & \tc{red}{\bu} \ar[l] \ar[d] & & & & & & & & & -2 \\
& & & \tc{red}{\bu} \ar@/^/[r]^{h_3} & \tc{red}{\bu} \ar[l] \ar[d] & \bu & & & & & & & \bu  & & -3 \\
& & \tc{red}{\bu} \ar@/^/[r]^{h_3} & \tc{red}{\bu} \ar[l] \ar[d] & \bu & & & & & & &  \bu & & & -4 \\
& & \da[ld] & \bu & & & & & & & \bu \da[dl] & & & & -5 \\
& & & C_2 & & & & & & & & C_3 & & &
}$$
\caption{ The complex $C_0=CFK^-(T(2,7))$, and the process of reducing to $C_3=\underline{CFK}^-(T(2,7))$. At each step, the map $h_i$ provides the chain homotopy necessary to cancel the horizontal arrows from $U^ny_i$ to $U^{n+1}x_i$ (the dots colored red). Multiplication by $U$ is always taken to be translation one down and one to the left, unless otherwise shown with a dotted arrow. }
\label{t27}
\end{figure}

The process described above explicitly obtains a reduced complex after a choice of an ordered, horizontally simplified basis. More generally, we have the following definition.

\begin{definition}
\label{reduced}
Let $K$ be a knot in $S^3$, and $C$ be a $(\Z,U)$-filtered chain complex. If $C$ is $(\Z,U)$-filtered chain homotopy equivalent to $CFK^-(K)$, and the differential on $C$ strictly decreases the filtration level, then $C$ is called the {\it{reduced $CFK^-(K)$}}, denoted $\underline{CFK}^-(K)$.
\end{definition}

The condition that the differential strictly decreases filtration level says precisely that there are no horizontal arrows. Of course, this complex is only well-defined up to $(\Z, U)$-filtered chain homotopy equivalence of complexes with no horizontal arrows.

Note that, since $CFK^-(K)$ is $\Z$-filtered by the Alexander grading, there is naturally a spectral sequence whose $(E_0,d_0)$ page is $(CFK^-(K),\del_H)$,  which converges to 
\begin{equation}
H_*(CFK^-(K),\del ) \cong
\begin{cases}
\F_{(-2i)}   &   i\geq 0, \\
0 &  \text{else}
\end{cases}
\end{equation}
as a graded group, where, as in Figure \ref{trefoil}, the subscript denotes the homological grading of each generator (the filtration on this group depends on $K$). An alternative view of the above construction, then, is that $\underline{CFK}^-(K)$ is the $E_1$ page of this spectral sequence. Further, the filtered endomorphism $U$ on $E_0$ induces a filtered endomorphism on $E_1$, and that is precisely the map $U$ we have defined on $\underline{CFK}^-(K)$. It is the additional information given by this induced $U$ map which will be useful for computing $d$-invariants in Section \ref{applications}.

\subsection{Sums of knots}
\label{sums}

As will be seen in this section and Section \ref{applications}, the reduced complex $\underline{CFK}^-(K)$ retains much of the information contained in $CFK^-(K)$, including, by definition, the homology of its associated graded complex, which is denoted $HFK^-(K)$. It was shown by Ozsv\'ath and Szab\'o in \cite[Theorem 7.1]{OSknot} that $CFK^-$ behaves simply under connected sums of knots; namely,
\begin{equation}
\label{connectedsumformula}
 CFK^-(K_1\# K_2) \cong CFK^-(K_1)\otimes_{\F [U]} CFK^-(K_2).
\end{equation}
However, these tensor product complexes are inconvenient to deal with by hand, even for sums of knots with small knot Floer homology. Many of the applications of this paper are to sums of knots, and so it will be convenient to be able to reduce a complex before taking a tensor product, in order to decrease the size of the product. The following theorem ensures that this is possible.
\begin{theorem}
\label{tensor}
If $K_1$ and $K_2$ are knots in $S^3$, then 
\begin{equation}
\underline{CFK}^-(K_1)\otimes_{\F [U]} CFK^-(K_2)
\label{reducedtensorproduct}
\end{equation}
 is a $(\Z,U)$-filtered chain deformation retract of $CFK^-(K_1\# K_2)$.
\end{theorem}
\begin{proof}[\sc Proof] After noting the relationship in equation \eqref{connectedsumformula}, the result will follow from Lemma \ref{retracttensor}. To verify that the lemma applies, however, we must verify the following. 

First, the complex $CFK^-(K_2)$ is, by definition, freely generated over $\F[U]$, and multiplication by $U$ decreases the filtration level by 1. Second, recall that the maps $U_i$ defined as in equation \eqref{U_i}, decrease the filtration level by at least 1. It follows that the map $U$ on the reduced complex $\underline{CFK}^-(K_1)$ also decreases the filtration level by at least 1. In other words, we define a filtration on the tensor product by formula \eqref{tensorfiltration}, and the result follows from Lemma \ref{retracttensor}.
\end{proof}

In order to use this effectively, we would like to further reduce the product complex \eqref{reducedtensorproduct}, which we will call $C$ for brevity, to get \( \underline{CFK}^-(K_1\# K_2) \). The method of reduction described explicitly above was for complexes which are freely generated over \( \F[U]\), but we should point out here that a complex such as $C$ can be handled similarly. This is because for sufficiently negative $i$, the subcomplex \[ A_i = \{ x \in C| A(x)\leq i \} \] is freely generated over \( \F [U]\), and the corresponding quotient \[ C/ A_i = \{ x\in C | A(x)>i \} \] is finitely generated over $\F$. So, after canceling the finite number of horizontal arrows with filtration level greater than $i$, we can use the same method as before.

\begin{remark}
\label{reducefirst}
Suppose $K_1$ and $K_2$ are knots in $S^3$, and we wish to reduce $CFK^-(K_1\# K_2)$. Then we can first reduce $CFK^-(K_1)$, tensor the reduced complex with $CFK^-(K_2)$, and then further reduce the product.
\end{remark}


To give an idea of how this facilitates computation, we consider a simple example, the sum of the right-handed and left-handed trefoils, $T(2,3)\#-~T(2,3)$. Figure \ref{trefoilsum} shows the knot Floer complexes of these two knots. To obtain $\underline{CFK}^-(T(2,3)\# -T(2,3))$, we first compute $\underline{CFK}^-(T(2,3))$ and then tensor it with $CFK^-(-T(2,3))$. To provide contrast, we also show in Figure \ref{trefoiltensor} the tensor product complex $CFK^-(T(2,3))\otimes_{\F[U]}CFK^-(-T(2,3))$.

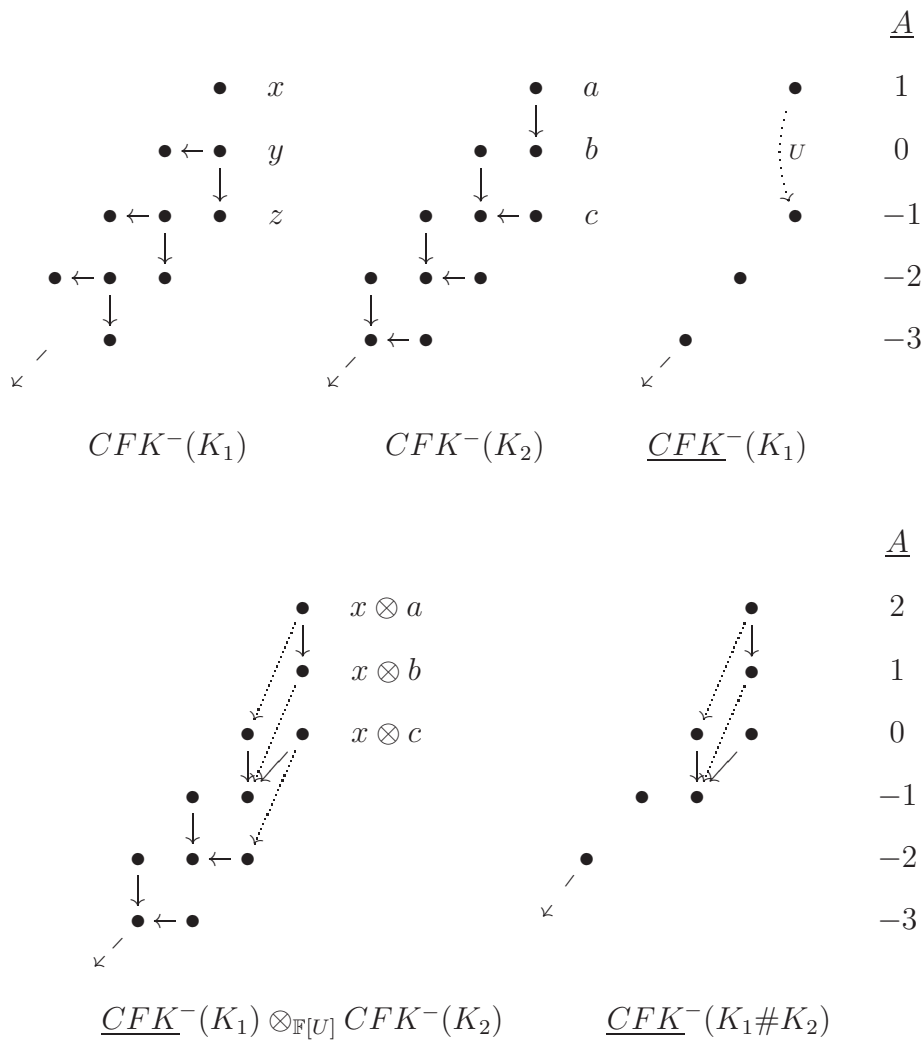
\begin{figure}
$$\xymatrixcolsep{0.65 pc}\xymatrixrowsep{0.65 pc}
\xymatrix{
 & & & & & & &               & & & & & &                      & & & &  \underline{A} \\
 & & & & \bu & x & &         & & & \bu \ar[d] & a & &     & & \bu \dar@/_/[dd]^U & & 1 \\
 & & & \bu & \bu \ar[l] \ar[d] & y & &   & & \bu \ar[d] & \bu & b & &          & & & & 0 \\
 & & \bu & \bu \ar[d] \ar[l] & \bu & z & &        & \bu \ar[d] & \bu & \bu \ar[l] & c & &      & & \bu & & -1 \\
 & \bu & \bu \ar[d] \ar[l] & \bu & & & &        \bu \ar[d] & \bu & \bu \ar[l] & & & &     & \bu & & & -2 \\
 & \da[dl] & \bu & & & & &                            \bu \da[dl] & \bu \ar[l] & & & & &          \bu \da[dl] & & & & -3\\
 & & & & & & &               & & & & & &                     & & & &  }$$
$$\xymatrixcolsep{0.5 pc}\xymatrixrowsep{0.5 pc}
\xymatrix{
 CFK^-(K_1) &  & &  &     CFK^-(K_2) & & &  \underline{CFK}^-(K_1) & }$$
$$\xymatrixcolsep{0.65 pc}\xymatrixrowsep{0.65 pc}
\xymatrix{
& & & & & & & & &                                 & & & & & & & & \\
& & & & & & & & &                                  & & & & & & & & \underline{A} \\
& & & & & & \bu \ar[d]  \dar[ddl] & x\otimes a & &                 & & & & & \bu \ar[d]  \dar[ddl] & & & 2 \\
& & & & & & \bu  \dar[ddl] & x\otimes b & &                           & & & & & \bu  \dar[ddl] & & & 1 \\
& & & & & \bu \ar[d] & \bu \ar[dl]  \dar[ddl] & x\otimes c & &                          &  & & & \bu \ar[d] & \bu \ar[dl] & & & 0 \\
& & & & \bu \ar[d] & \bu & & & &                                      & & & \bu & \bu & & & & -1 \\
& & & \bu \ar[d] & \bu & \bu \ar[l] & & & &                       & & \bu \da[dl] & & & & & & -2 \\
& & & \bu \da[dl] & \bu \ar[l] & & & & &                            & & & & & & & & -3 \\
& & & & & & & & &                                   & & & & & & & & }$$
$$ \xymatrixcolsep{0.5 pc}\xymatrixrowsep{0.5 pc}
\xymatrix{
& & & \underline{CFK}^-(K_1)\otimes_{\F[U]}CFK^-(K_2) & & &  \underline{CFK}^-(K_1\# K_2) & & & & & & & & & & &
}$$
\caption{An example of a connected sum, where $K_1$ is the right-handed trefoil and $K_2$ is the left-handed trefoil. We can first reduce to get $\underline{CFK}^-(K_1)$, then tensor this complex with $CFK^-(K_2)$. One more reduction gives the complex $\underline{CFK}^-(K_1\# K_2)$. }
\label{trefoilsum}
\end{figure}

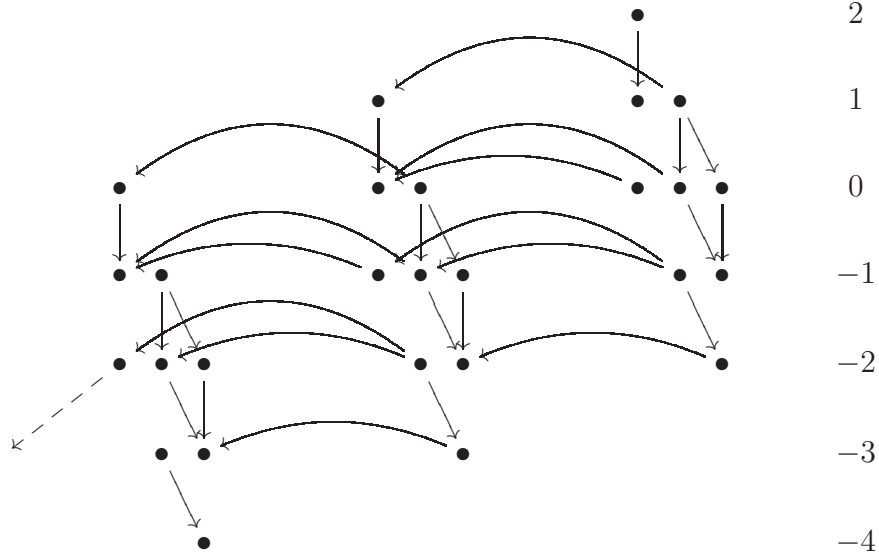
\begin{figure}
$$\xymatrixcolsep{0.25 pc} \xymatrixrowsep{1.5 pc} 
\xymatrix{
& & & & & & &                     & & & & &                                     & & &                                         & & &  & &                                \bu \ar[d] & & &           & & & 2\\
& & & & & & &                     & & &  & &                                  \bu \ar[d] & & &                                    & & &   & &                               \bu & \bu \ar[dr] \ar@/_2pc/[lllllllll] \ar[d] & &     & & & 1 \\
& & & & \bu \ar[d] & & &                 & & &  & &                    \bu & \bu \ar[dr] \ar@/_2pc/[lllllllll] \ar[d] & &                              & & & & &                       \bu \ar@/_1pc/[llllllll] & \bu \ar[dr] \ar@/_2pc/[lllllllll] & \bu \ar[d] &    & & & 0 \\
& & & & \bu & \bu \ar[dr] \ar[d] & &        & & &  & &                        \bu \ar@/_1pc/[llllllll] & \bu \ar@/_2pc/[lllllllll] \ar[dr] & \bu \ar[d] &                        & & &   & &                        & \bu \ar@/_1pc/[llllllll] \ar@/_2pc/[lllllllll] \ar[dr] & \bu &         & & & -1 \\
& & & & \bu \da[lllld] & \bu \ar[dr] & \bu \ar[d] &  & & &  & &                         & \bu \ar@/_1pc/[llllllll] \ar@/_2pc/[lllllllll] \ar[dr] & \bu &                               & & &  & &                       & & \bu \ar@/_1pc/[llllllll] &               & & & -2 \\
& & & &  & \bu \ar[dr] & \bu &        & & &  & &                                  & & \bu \ar@/_1pc/[llllllll] &                                     & & &  & &                        & & &                      & & & -3 \\
& & & & & & \bu &              & & &  & &                                  & & &                                           & & &   & &                       & & &                       & & & -4 
}$$
\caption{The tensor product complex corresponding to the sum of the right- and left-handed trefoils. Even for knots with simplest nontrivial knot Floer complexes, computing with tensor product complexes becomes tedious.}
\label{trefoiltensor}
\end{figure}

\section{Applications to $L$-space knots}
\label{applications}

In this section, we show how the reduced complex $\underline{CFK}^-(K)$ can be used to elucidate some properties of $L$-space knots. Recall that a rational homology 3-sphere $Y$ is called an {\it{$L$-space}} if -- like a lens space -- it has the \textquotedblleft smallest possible" Heegaard Floer homology; i.e., for each spin$^{c}$ structure $\mf{t}$, $\widehat{HF}(Y,\mf{t} )\cong\F$. A knot $K$ in $S^3$ is called an {\it{$L$-space knot}} if $n$-surgery on $S^3$ along $K$ is an $L$-space, for some positive integer $n$. It was shown in \cite[Theorem 1.2 and Corollary 1.6]{OSlens}, and restated more conveniently for our purposes in \cite[Remark 6.6]{Homconcordance}, that $L$-space knots have knot Floer complexes of a particular form, which we describe here.

\begin{prop}[Ozsv\'ath-Szab\'o]
\label{lspacechar}
If $K$ admits a positive $L$-space surgery, then $CFK^-(K)$ has a basis $\{ x_{-k}, \cdots, x_k \}$ with the following properties:
\begin{itemize}
\item $A(x_i)=n_i$, where $n_{-k}<n_{-k+1} <\cdots < n_{k-1}<n_k$ \\
\item $n_i = -n_{-i}$\\
\item If $i \equiv k \mod 2$, then $\del(x_i)=0$\\
\item If $i \equiv k+1 \mod 2$, then $\del(x_i) = x_{i-1} + U^{n_{i+1}-n_i}x_{i+1}$ \qed
\end{itemize}
\end{prop}

Diagramatically, the knot Floer complex of an $L$-space knot has a \textquotedblleft staircase" shape, as shown, for example, in Figure \ref{t34}. The basis described in Proposition \ref{lspacechar} is, in particular, horizontally simplified, and so the method of reduction will proceed exactly as in Section \ref{knotfloer}. We include a proof of the corollary below, although the result should be more readily evident by seeing the reduction in Figure \ref{t34}.

\begin{figure}
$$\xymatrixcolsep{0.75 pc} \xymatrixrowsep{0.75 pc}
\xymatrix{
& & & & &            & & &      & & &                                              & & & &            &                                           & & \underline{A} \\
& & & & \bu & x_2      & &  &    & & & \bu \dar@/_/[ddd]^U          & &  & &            & \bu \dar@/_/[ddd]^U       & &  3 \\
& & & \bu & \bu \ar[l] \ar[dd] & x_1   &  & &      & & &                    & & &  &            &                                          & & 2 \\
& & \bu & \bu \ar[l] \ar[dd] & &    & &  &     & & &                    & & &  &            &                                          & & 1 \\
& \bu & \bu \ar[l] \ar[dd] & & \bu & x_0    & & &    & & & \bu        & & & &            & \bu                                    & & 0 \\
\bu & \bu \ar[l] \ar[dd] & & \bu & &     & & &    & & \bu &        & & & &             \bu \dar@/_/[ddr]^U &      & & -1 \\
& & \bu & & \bu \ar[ll] \ar[d] & x_{-1}            & & &    & \bu & & \bu \ar[ll] \ar[d]    & & & &     &                          & & -2 \\
& \bu \da[dl] & & \bu \ar[ll] \ar[d] & \bu & x_{-2}      & & &     \bu & & \bu \ar[ll] \ar[d] & \bu   & & & &    & \bu                    & & -3 \\
& & & \bu & &        & &  &                              & \da[dl] & \bu &            & & & &             \bu \da[dl] &                                  & & -4 \\
& & & &              & & &                                & & &                                 & & & &                  &                           & & & -5
}$$
$$\xymatrix{
& CFK^-(T(3,4)) & & & & & \underline{CFK}^-(T(3,4)) &
}$$
\caption{At the left is the staircase-shaped knot Floer complex for the (3,4)-torus knot, and at the right is its reduction (with an intermediate step shown in between). Each time we cancel the horizontal arrow from $x_{2-2i-1}$ to $x_{2-2i}$, we see that $U$ takes the last remaining power of $x_{2-2i}$ to $x_{2-2i-2}$, so, in fact, every generator which remains in $\underline{CFK}^-(T(3,4))$ can be written as $U^kx_2$ for some $k\geq 0$.}
\label{t34}
\end{figure}
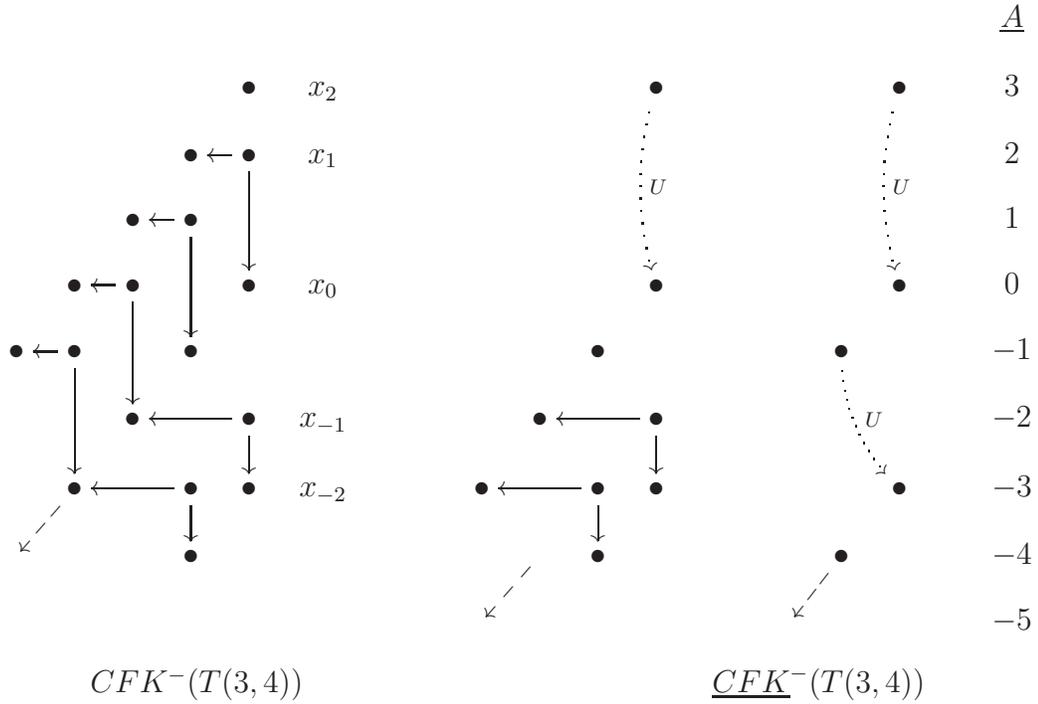

\begin{corollary}
\label{lspacered}
If $K$ is an $L$-space knot, then $\underline{CFK}^-(K)$ has exactly one generator of Maslov grading $-2i$ for each $i\geq 0$, and no other generators. Further, if $x$ is the generator with Maslov grading $-2i$, and $y$ is the generator with Maslov grading $-2i-2$, then $Ux=y$.
\end{corollary}
\begin{proof}[\sc Proof]
We will first show that if we use a basis as in Proposition \ref{lspacechar}, and proceed with the reduction as described in Section \ref{knotfloer}, then we get a representative of \( \underline{CFK}^-(K)\) with the desired properties. We will then show that, in fact, any representative must be isomorphic to this one.

First, we consider the subquotient complex \[ \widehat{CFK}(K) := C\{i=0\}. \] This is a filtered complex which is chain homotopy equivalent to $\widehat{CF}(S^3)$, and so its homology is generated by a single element. To fix an absolute Maslov grading, this generator of the homology of $\widehat{CF}(S^3)$ is declared to have Maslov grading zero. It is clear from the explicit differential given in Propostion \ref{lspacechar} that, when $K$ is an $L$-space knot, $x_k$ generates the homology of $\widehat{CFK}(K)$, and so $M(x_k)=0.$

With this as our starting point, we now consider what happens through reduction. To simplify notation, we will define \[r_i := n_i-n_{i-1} \] to be the difference in the Alexander gradings of $x_i$ and $x_{i-1}$.  We first cancel the horizontal arrows which form the top steps of the staircases, those corresponding to \[ \del_H(U^mx_{k-1}) = U^{m+r_k}x_k\] for $m\geq 0$. Note that after canceling $U^{r_k}x_k$, we have 
\begin{align*}
U_1(U^{r_k-1}x_k) =& \ \left( \pi\circ (U-\del h_1 U)\circ\iota\right) (U^{r_k-1}x_k) \\
=& \ \pi\circ (U^{r_k}x_k-\del h_1 (U^{r_k}x_k)) \\
=& \ \pi\circ (U^{r_k}x_k- \del x_{k-1}) \\
=& \ x_{k-2}
\end{align*}
That is, after canceling higher $U$-powers of $x_k$, the map $U$ takes the highest remaining power, $U^{r_k-1}x_k$, to the next generator, $x_{k-2}$. The exact same argument applies to the cancellation of the horizontal arrows \[\del_H (x_{k-3}) = U^{r_{k-2}}x_{k-2}, \] and we see that \[ U_2(U^{r_{k-2}-1}x_{k-2}) = x_{k-4}.\] We proceed in this fashion, until finally we see that \[ U_k(U^{r_{-k+2}-1}x_{-k+2}) = x_{-k}.\] It follows that each of these generators of $\underline{CFK}^-(K)$ can be written as \[U^ix_k \] for some $i\geq 0$. Since multiplication by $U$ lowers the Maslov grading by 2, the result follows.

Now let us call the reduced complex just constructed $C$, and suppose that $C'$ is a $(\Z,U)$-filtered chain homotopy equivalent complex which also has no horizontal arrows (i.e., a different representative of \( \underline{CFK}^-(K)\)). Since it is filtered chain homotopy equivalent, each subquotient complex \( C'\{j=k\} \) must have homology isomorphic to that of \( C\{j=k\} \), which is either 0 or $\F$, depending on $k$. Since there are no horizontal arrows, these subquotient complexes have trivial differential, so they either have no generators or they have exactly 1 generator, with even homological grading. Therefore, $C'$ also has trivial differential, so it is isomorphic to $C$ as a filtered chain complex. The fact that the equivalence is $(\Z, U)$-filtered implies that $U$ also takes generator to generator for $C'$ as it does for $C$, so in fact they are isomorphic as $(\Z,U)$-filtered complexes. \qedhere 
\end{proof}
A concise way of stating Corollary \ref{lspacered} is that, for an $L$-space knot $K_1$, 
\begin{equation}
\label{lspacefu}
\underline{CFK}^-(K_1)\cong \F[U]_{(0)},
\end{equation}
where the subscript here means that the generator has Maslov grading zero. Necessarily, this complex also has trivial differential.
This means that a tensor product of the form 
\begin{equation}
\label{lspacetensorfu}
\underline{CFK}^-(K_1)\otimes_{\F[U]} CFK^-(K_2)
\end{equation}
will be isomorphic to $CFK^-(K_2)$ as a chain complex, which we will make use of below. It is important to note however, that the isomorphism \eqref{lspacefu} is not filtered (when $K_1$ is not the unknot). As a result, the tensor product \eqref{lspacetensorfu} is not $(\Z,U)$-filtered chain homotopy equivalent to $CFK^-(K_2)$.


Knowing that $L$-space knots must have this particularly simple reduced knot Floer complex, we make use of the behavior under connected sums to record the following observation.

\begin{theorem}
\label{prime}
A knot in $S^3$ which admits an $L$-space surgery must be a prime knot.
\end{theorem}

The general argument, a proof by contradiction using Corollary \ref{lspacered}, will be suggested by considering an example, so we refer back to Figure \ref{trefoilsum}. Consider the generator of maximal Alexander grading in $CFK^-(K_1)$, which we labeled $x$. We have $A(x)=1$, and so \[A(Ux)=0=A(x)-1.\] But $Ux$ is canceled by a horizontal arrow, so when we move to $\underline{CFK}^-(K_1)$, we have \[A(Ux)=-1 < A(x)-1\] That is, multiplication by $U$ \textquotedblleft jumps" down by more than 1. When we take the tensor product \[\underline{CFK}^-(K_1)\otimes_{\F[U]}CFK^-(K_2), \] this has the effect of \textquotedblleft bending downward" the horizontal arrow from $c$ to $Ub$. That is, the arrow from $x\otimes c$ to $Ux\otimes b$ is not horizontal. Therefore, when we reduce this tensor product complex, we cannot cancel these generators, and so they both remain. But, being connected by an arrow, their Maslov gradings differ by exactly 1, so the complex $\underline{CFK}^-(K_1\# K_2)$ does not have the form described in Corollary \ref{lspacered}. With this example as motivation, we provide the details.

\begin{proof}[\sc Proof of Theorem \ref{prime}]
We begin by noting that if {\it{negative}} surgery on a knot produces an $L$-space, then positive surgery on its mirror image produces an $L$-space, so by definition its mirror image is an $L$-space knot. Since a knot is a nontrivial connected sum if and only if its mirror image is, it will be sufficient to show that {\it{positive}} surgery on a connected sum can never produce an $L$-space. That is, we will show that no nontrivial connected sum is an $L$-space knot.

Suppose $K_1$ and $K_2$ are two nontrivial knots in $S^3$. Let $C_i$ denote the complex $CFK^-(K_i)$. Choose a horizontally simplified basis for $C_1$, where $y_1$ and $x_1$ are generators such that $\del_H(y_1)=U^{r_1}x_1$. Likewise, choose a horizontally simplified basis for $C_2$, where $Y_1$ and $X_1$ are generators such that $\del_H(Y_1)=U^{R_1}X_1$. Without loss of generality, assume that $r_1\leq R_1$. Our goal is to show by contradiction that $\underline{CFK}^-(K_1\# K_2)$ is a complex which cannot correspond to an $L$-space knot, by Corollary \ref{lspacered}.

We will make use of Remark \ref{reducefirst}, and begin by reducing $C_1$ to get a complex $\underline{C_1}$. We will denote the filtration on this reduced complex by $F_1$ and the filtration on $C_2$ by $F_2$.  Finally, the filtration on the tensor product $\underline{C_1}\otimes_{\F[U]} C_2$, defined as in equation \eqref{tensorfiltration}, will be denoted $F$.

After reducing $C_1$, we have $$F_1(U^{r_1-1}x_1)= F_1(x_1)-(r_1-1),$$ but the map $U$ on $U^{r_1-1}x_1$ lowers the filtration level by at least two, so $$F_1(U^{r_1}x_1) < F_1(x_1)-r_1.$$ As a consequence, since $R_1\geq r_1$, and $U$ always lowers the filtration level by at least 1, $$F_1(U^{R_1}x_1) < F_1(x_1)-R_1.$$ Note also that, on the freely generated complex $C_2$, $U$ is a homogeneous map of degree 1, so $$F_2(Y_1)=F_2(X_1)-R_1.$$ It follows that
\begin{align*}
F(U^{R_1}x_1\otimes X_1) =&\ F_1(U^{R_1}x_1) + F_2(X_1)\\
<&\ F_1(x_1)-R_1 + F_2(X_1)\\
=&\ F_1(x_1)+F_2(Y_1)\\
=&\ F(x_1\otimes Y_1).
\end{align*}
Because of this, every term in \[\del (x_1\otimes Y_1) = \del x_1\otimes Y_1 + x_1\otimes \del Y_1 \] has filtration level strictly less than that of $x_1\otimes Y_1$. That is to say, $\del_H(x_1\otimes Y_1) = 0$. When we reduce the tensor product complex, there is no horizontal differential to cancel $x_1\otimes Y_1$, so it will project to a nonzero homogeneous element in $\underline{CFK}^-(K_1 \# K_2)$. Of course, $x_1\otimes X_1$ also projects to a nonzero homogeneous element in
$\underline{CFK}^-(K_1 \# K_2)$. But,
\begin{align*}
M(x_1\otimes Y_1) =&\ M(x_1)+M(Y_1)\\
=&\ M(x_1) + M(U^{R_1}X_1)+1 \\
=&\ M(x_1) + M(X_1)+1-2R_1 \\
=&\ M(x_1\otimes X_1)+1-2R_1,
\end{align*}
so $\underline{CFK}^-(K_1 \# K_2)$ has two elements with Maslov gradings of opposite parity. By Corollary \ref{lspacered}, $K_1 \# K_2$ cannot be an $L$-space knot. 
\end{proof}

We now turn to our second application, pertaining to the Heegaard Floer correction terms, or $d$-invariants. Given a rational homology three-sphere $Y$ and spin$^{c}$ structure $\mf{t}$, we obtain a chain complex $CF^{\infty}(Y,\mf t)$ which is freely generated over $\F[U,U^{-1}]$, as described in Section \ref{knotfloer}, and its associated subcomplex $CF^-(Y,\mf t)$. The homology of this subcomplex, denoted $HF^-(Y,\mf t)$, consists of a direct summand isomorphic to $\F [U]$, and possibly other terms which are $U$-torsion. The correction term associated to $(Y,\mf t)$, denoted $d(Y,\mf t)$, is simply the maximal Maslov grading of any nontorsion generator in $HF^-(Y,\mf t)$. 

 Given a knot $K$ in $S^3$, let $S^3_1(K)$ denote the integer homology sphere obtained from $S^3$ by doing Dehn surgery along $K$ with slope 1. We can associate to this manifold a number, $d(S_1^3(K),\mf{t})$ (where $\mf{t}$ is the unique spin$^c$ structure on $S_1^3(K)$), which we will abbreviate as $d_1(K)$. This invariant was studied by Peters in \cite{Peters}, where it was shown to have the following properties.

\begin{prop}[Theorem 1.5 and Proposition 2.1 of \cite{Peters}] 
\label{d1prop}
For any knot $K$ in $S^3$,
\begin{itemize}
\item $d_1(K)$ is an even integer\\
\item $d_1(K)$ is a concordance invariant of $K$\\
\item  If we denote by $g_4(K)$ the smooth four-dimensional genus of $K$, \[ 0\leq -d_1(K) \leq 2g_4(K). \] 
\end{itemize}
\end{prop}

In addition, Peters gave an algorithm to compute $d_1(K)$ from $CFK^{\infty}(K)$, using the fact that $CFK^{\infty}(K)$ contains all the information needed to compute the Heegaard Floer homology of manifolds arising from surgery on $K$. We briefly recount the idea here. For details, see \cite{Peters}.

In \cite[Lemma~7.11]{OSabsgr}, the degrees of the maps in the integer surgery exact sequence \[\cdots \rightarrow HF^+(S^3_0(K))\rightarrow HF^+(S^3_N(K)) \rightarrow HF^+(S^3) \rightarrow \cdots\] were computed, from which it was concluded in \cite[Sec.~5]{Peters} that\footnote{For three-manifolds with $H_1(Y)\cong \Z $, Ozsv\'ath and Szab\'o define $d_{\pm 1/2}(Y)$ to be the minimal grading of an element in $HF^+(Y,\mf{s_0})$ which is in the image of $U^k$ for all $k>0$ whose grading is additionally congruent to $\pm 1/2 \mod 2$, where $\mf{s_0}$ is the unique spin$^c$-structure for which $c_1(\mf{s_0})=0$.} 
\begin{equation}
d_{1/2}(S^3_0(K))=d(S^3_N(K),\mathfrak{s}_0)-\frac{N-3}{4}.
\label{d0n}
\end{equation}
In particular, we have
\begin{equation}
d_1(K) = d(S^3_N(K),\mathfrak{s}_0) - \frac{N-1}{4},
\label{d1n}
\end{equation}
so the invariant $d_1(K)$ is determined by $d(S_N^3(K),\mathfrak{s}_0)$. For $N$ sufficiently large, this can be computed directly from $CFK^{\infty}(K)$.

Let $A_0^+$ denote the quotient complex \[ C\{i\geq 0 \text{ or } j\geq 0\}\] of $CFK^{\infty}(K)$ (recall that $C\{ S\}$ denotes the elements with $(i,j)$-coordinates in $S$, and the arrows between these elements). Ozsv\'ath and Szab\'o \cite[Corollary 4.2]{OSknot} (c.f., \cite{Rasmussenknot}) show that, for any sufficiently large positive integer $N$,
\begin{equation}
\label{shift}
HF^+_{l+\left(\frac{N-1}{4}\right)}(S^3_N(K),\mf{s}_0) \cong H_l (A_0^+ ).
\end{equation}
That is, up to a shift in grading which depends on $N$, the homology of this complex is the Heegaard Floer homology of the three-manifold obtained by surgery. Combining equations \eqref{d1n} and \eqref{shift}, we see that the grading shifts cancel nicely, and $d_1(K)$ is equal to the minimum grading of a generator of $H_*(A_0^+)$ which is in the image of $U^k$ for all $k>0$.

\begin{remark}
It should be pointed out that, by \eqref{d0n}, we could get the same information from the invariant $d(S^3_N(K),\mf{s}_0)$, which is also a concordance invariant, as we get from $d_1(K)$. The choice $N=1$ is a matter of convenience, because it gives a four-genus bound without any shift.
\end{remark}

We will find it convenient to work with the subcomplex $CFK^-(K)$ rather than the quotient $CFK^+(K)$. From this point of view, $d_1(K)$ is the {\it{maximum}} grading of a non-torsion generator of homology of the subcomplex \[C\{i\leq 0 \ \ \text{and} \ \ j\leq 0\}. \] 

\begin{remark}
To justify this this point of view, we first point out that Ozsv\'ath and Szab\'o define $d^-(Y,\mf t)$ to be the maximal grading of a non-torsion generator in $HF^-(Y,\mf t)$, and observe in the proof of \cite[Prop.~4.2]{OSabsgr} that \[d^-(Y,\mf t) = d(Y,\mf t) -2\] (recalling that $U$ lowers grading by 2). But our definition of $CF^-$ differs from Ozsv\'ath and Szab\'o's by a shift by $U^{-1}$ (see Remark \ref{convention}). So, the maximal grading of a non-torsion generator of our $CF^-(Y,\mf{t})$ is \[d'(Y,\mf t) = d^-(Y,\mf t) + 2 = d(Y, \mf t), \] so we will think of the $d$-invariant this way.
\end{remark}

Figure \ref{d1T(2,5)} shows how $d_1$ can be computed from the knot Floer complex in the case of the (2,5)-torus knot.

\begin{figure}
$$\xymatrixcolsep{0.25 pc}\xymatrixrowsep{0.75 pc}
\xymatrix{
& & & & & & \!\!\!\! (0) \!\!\!\! & \bu & \ \ & \bu \ar[ll] \ar[dd] & & & & 2 \\
& & & & & & & & & & \da[ru] & & & \\
& & & & \! \!\!\! (-2)\!\!\! \! & \bu & & \bu \ar[ll] \ar[dd] & & \bu & \ \  & \bu \ar[ll] \ar[dd] & & 1 \\
& & & & & & & & \ar@{--}[llllllll] \ar@{--}[dddddddd] & & & & & \\
& & \!\!\!\! (-4)\!\!\!\! & \bu & & \bu \ar[ll] \ar[dd] & & \bu & & \bu \ar[dd] \ar[ll] & & \bu & & j=0 \\
& & & & & & \ar@{.}[llllll] \ar@{.}[dddddd] & & & & & & &  \\
(-6)\!\!\!\! & \bu & & \bu \ar[ll] \ar[dd] & & \bu & & \bu \ar[ll] \ar[dd] & & \bu & & & & -1 \\
& & & & & & & & & & & & & \\
& & & \bu & & \bu \ar[ll] \ar[dd] & & \bu & & & & & & -2 \\
& & \da[dl] & & & & & & & & & & & \\
& & & & & \bu & & & & & & & & -3 \\
& & & & & & & & & & & & & \\
& -3 & & -2 & & -1 & & i=0 & & 1 & & 2 & 
}$$
\caption{The complex $CFK^{\infty}(T(2,5))$. The invariant $d_1(T(2,5))$ can be seen to equal $-2$, either by considering the maximal grading of a generator of homology of the subcomplex $C\{i\leq 0  \ \text{and}  \ j\leq 0\}$ (below and to the left of the dashed line); or by considering the minimal grading of a generator of homology of the quotient complex $C\{i\geq 0  \ \text{or} \ j\geq 0\}$ (above or to the right of the dotted line). We will find it easier to work with the subcomplex through the rest of this paper.}
\label{d1T(2,5)}
\end{figure}
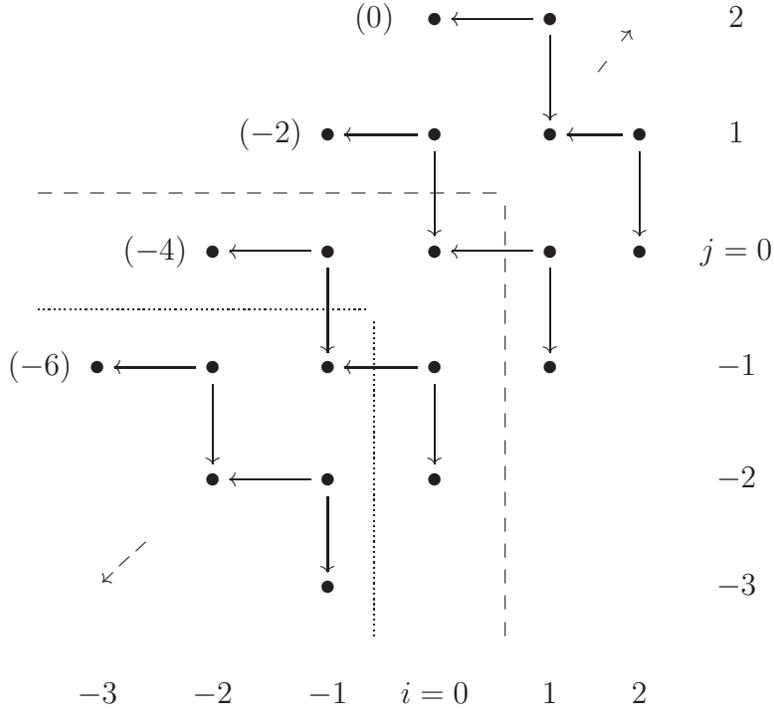

There is another concordance invariant which comes from the knot Floer complex, $\tau(K)$, which was introduced by Ozsv\'ath and Szab\'o in \cite{OSfourballgenus}, where they showed that it also gives a lower bound for the smooth four-dimensional genus of $K$, \[|\tau(K)| \leq g_4(K).\]
Given a knot Floer complex, this invariant is easily computed, and yet has been shown to be a quite powerful four-genus bound. For example, its value on torus knots, shown in \cite[Corollary 1.7]{OSfourballgenus} to be 
\begin{equation}
\tau(T(p,q)) = \frac{(p-1)(q-1)}{2},
\label{tautorus}
\end{equation}
was used to provide an alternate proof of the Milnor conjecture, which says that this is in fact the four-genus of $T(p,q)$. The invariant is defined from $CFK^{\infty}(K)$ by considering the subquotient complex \[ \widehat{CFK}(K):=C\{i=0\}. \] If we let $\iota_k$ be the inclusion map \[\iota_k: C\{i=0,j\leq k\} \to C\{i=0\},\] we get an induced map on homology
\[(\iota_k)_*: H_*(C\{i=0,j\leq k\}) \to H_*(C\{i=0\}).\] This map is clearly an isomorphism for large enough $k$, and the zero map for sufficiently negative values of $k$ (since the complex is finitely generated). We can then define \[\tau(K):= \min \{k | (\iota_k)_* \ \text{is non-trivial} \}. \] This quantity is additive under tensor products of complexes, and therefore $\tau$ defines a homomorphism from the smooth concordance group to $\Z$. 

It follows from Proposition \ref{lspacechar} that, for an $L$-space knot $K$, \[\tau(K) = A(x_k) = \max\{ j | \widehat{HFK}(K,j)\neq 0\},\] (which is also the Seifert genus of $K$). In general, it was shown in \cite{OSknot} that
\begin{equation}
\sum_j \chi \left(\widehat{HFK}(K,j)\right)\cdot T^j = \Delta_K(T),
\label{alexander}
\end{equation}
where $\Delta_K(T)$ is the symmetrized Alexander polynomial of $K$. Since, for an $L$-space knot, we can choose a basis for which the rank of $\widehat{CFK}(K,j)$ is either 0 or 1 for each $j$, the rank of each subcomplex is determined by its Euler characteristic, so, by Proposition \ref{lspacechar}, the knot Floer complex contains the same amount of information as the Alexander polynomial.

In particular, $\tau(K)=\deg \Delta_K(T)$. That, however, is all the information $\tau$ can give in this case. The statement that $\tau(K_1\#-K_2)=0$ for two $L$-space knots $K_1$ and $K_2$ is precisely the statement that their Alexander polynomials have equal degree. In contrast, the next theorem gives  a sense in which the invariant $d_1$ is more sensitive.

\begin{theorem}
\label{alexanderconcordance}
Suppose that $K_1$ and $K_2$ are two knots in $S^3$ which admit positive $L$-space surgeries. If $$d_1(K_1\#-K_2)=d_1(-K_1\#K_2)=0,$$ then $$\Delta_{K_1}(T)=\Delta_{K_2}(T).$$ In particular, the Alexander polynomial is a concordance invariant of $L$-space knots.
\end{theorem}
\begin{proof}[\sc Proof]
As mentioned above, in light of Proposition \ref{lspacechar}, the Alexander polynomial of $K_i$ gives equivalent information to the knot Floer complex of $K_i$, which we will represent by its staircase shape. We can represent a staircase by listing the horizontal lengths in order from left to right. By the symmetry of the Alexander polynomial, this list is also the list of vertical lengths, in order from bottom to top. Suppose that $K_1$ has staircase $\{ \alpha_1, \alpha_2, \ldots , \alpha_n\}$, and $K_2$ has staircase $\{ \beta_1, \beta_2, \ldots , \beta_m\}$, as shown in Figure \ref{labeledstaircases}, and also that $$d_1(K_1\#-K_2)=d_1(-K_1\#K_2)=0.$$ The proof will proceed by showing first that the Alexander polynomials must have equal degrees, then, one step at a time, that $\alpha_i = \beta_i$ for all $i$ (and consequently, that $m=n$). 

\begin{figure}
$$\xymatrixcolsep{1.5 pc}\xymatrixrowsep{1.5 pc}
\xymatrix{
\bu & \bu \ar[l]^{\alpha_1} \ar[d]^{\alpha_n} & &  &  & \bu & \bu \ar[l]^{\beta_1} \ar[d]^{\beta_m} & & &  \\
& \bu & \bu \ar[l]^{\alpha_2} \ar@{..}[dr] &  &  &   & \bu & \bu \ar[l]^{\beta_2} \ar@{..}[dr]  &  &       \\
&  &  & \bu & \bu \ar[l]^{\alpha_n} \ar[d]^{\alpha_1} &    &  &  & \bu & \bu \ar[l]^{\beta_m} \ar[d]^{\beta_1}      \\
&  &  &  & \bu &    &  &  &  &   \bu    \\
& & K_1 & & & & & K_2 & & 
}$$
\caption{The staircases for $K_1$ and $K_2$.}
\label{labeledstaircases}
\end{figure}
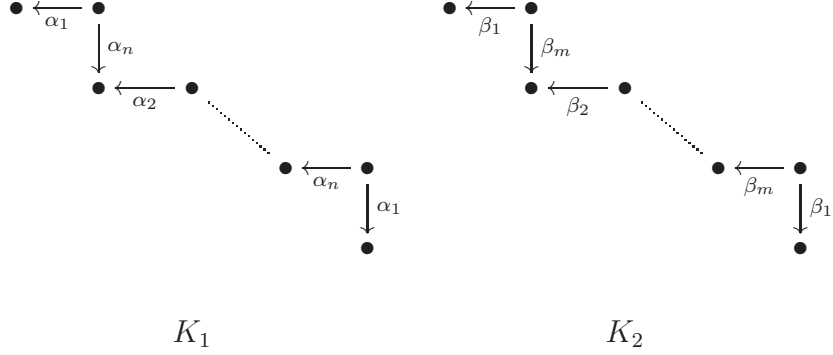

The complex $CFK^-(-K_2)$ will take the shape of an \textquotedblleft upside down staircase", since $-K_2$ is the mirror image of an $L$-space knot (see Figure \ref{mirrorstaircase}). In this case, $\tau(-K_2) = - \deg \Delta_{K_2}(T)$. We record here a particularly useful property of these complexes, which follows from direct inspection.

\begin{remark}
\label{upsidedown}
If $K$ is an $L$-space knot, then the knot Floer complex of its mirror, $CFK^-(-K)$, has a basis for which it satisfies the following:
\begin{itemize}
\item $CFK^-(-K)$ splits into a direct sum of complexes $C_{2i}$, for each integer $i$, where $C_{2i}$ consists of the homogeneous elements of Maslov gradings $2i$ and $2i-1$\\
\item for $i>0$, the complex $C_{2i}$ is acyclic\\
\item for $i\leq 0$, the complex $C_{2i}$ has homology isomorphic to $\F$, generated by the sum of all homogeneous elements of Maslov grading $2i$
\end{itemize}
\end{remark}

\begin{figure}
$$\xymatrixcolsep{1.5 pc}\xymatrixrowsep{1.5 pc}
\xymatrix{
 \bu & \bu \ar[l]^{\beta_1} \ar[d]^{\beta_m} & & &                        & \bu \ar[d]^{\beta_1} & & & &  \\
   & \bu & \bu \ar[l]^{\beta_2} \ar@{..}[dr]  &  &                         & \bu & \bu \ar[l]^{\beta_m} \ar@{..}[dr] & & &    \\
&  &  & \bu & \bu \ar[l]^{\beta_m} \ar[d]^{\beta_1}                   & & & \bu & \bu \ar[l]^{\beta_2} \ar[d]^{\beta_m} &    \\
   &  &  &  &   \bu                                                                            & & & & \bu & \bu \ar[l]^{\beta_1}    \\
 & & K_2 & &  & & & -K_2 & & 
}$$
\caption{The knot Floer complex of the $L$-space knot $K_2$, and that of its mirror image.}
\label{mirrorstaircase}
\end{figure}
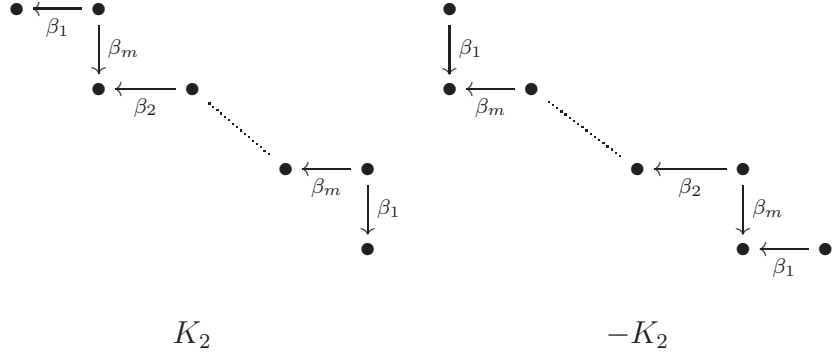

If $K_1$ is an $L$-space knot, then, by Corollary \ref{lspacered}, the complex $\underline{CFK}^-(K_1)$ is \textquotedblleft almost" the knot Floer complex of the unknot. More precisely, each is isomorphic to $\F[U]$, supported in grading zero, with the only difference being that, for the unknot, $U$ is homogeneous of degree one, while for $K_1$, $U$ is a non-homogeneous map which decreases the Alexander grading by {\it{at least}} one.  It follows that the tensor product complex $$\underline{CFK}^-(K_1)\otimes_{\F[U]} CFK^-(-K_2),$$ which we will denote by $C$, is \textquotedblleft almost" an upside down staircase; the only difference being that some of its \textquotedblleft stairs" have been bent. Figure \ref{t27-t34} shows an example of this, the complex $\underline{CFK}^-(T(2,7))\otimes_{\F[U]} CFK^-(-T(3,4))$ (recall that the reduced complex $\underline{CFK}^-(T(2,7))$ is shown in Figure \ref{t27}).

\begin{figure}
$$\xymatrixcolsep{0.6 pc} \xymatrixrowsep{0.6 pc}
\xymatrix{
& & & & & \bu \ar[d] & 3 \\
& & & & \bu \ar[d] & \bu & 2 \\
& & & \bu \ar[d] & \bu & & 1 \\
& & \bu \ar[d] & \bu & & \bu \ar[ll] \ar[dd] & 0 \\
& \bu \ar[d] & \bu & & \bu \ar[dd] \ar[ll] & & -1 \\
\bu \ar[d] & \bu & & \bu \ar[dd] \ar[ll] & & \bu & -2 \\
\bu & & \bu \ar[dd] \ar[ll] & & \bu & \bu \ar[l] & -3 \\
& \da[ld] & & \bu & \bu \ar[l] & & -4 \\
& & \bu & \bu \ar[l] & & & -5
 } \ \ \ \xymatrix{
& & & & & & & \bu \ar[d] & 6 \\
& & & & & & & \bu & 5 \\
& & & & & & \bu \ar[d] & & 4 \\
& & & & & & \bu & \bu \ar[lldd] \ar[dd] & 3 \\
& & & & & \bu \ar[d] & & & 2 \\
& & & & & \bu & \bu \ar[lldd] \ar[dd] & \bu & 1 \\
& & & & \bu \ar[d] & & & \bu \ar[ld] & 0 \\
& & & \bu \ar[d] & \bu & \bu \ar[lld] \ar[dd] & \bu & & -1 \\
& & \bu \ar[d]  & \bu & & & \bu \ar[ld] & & -2 \\
& \bu \ar[d] & \bu & & \bu \ar[ll] \ar[dd] & \bu & & & -3 \\
\bu \ar[d] & \bu & & \bu \ar[ll] \ar[dd] & & \bu \ar[ld] & & & -4 \\
\bu & & \bu \ar[ll] \ar[dd] & & \bu & & & & -5 \\
& \da[ld] & & \bu & \bu \ar[l] & & & & -6 \\
& & \bu & \bu \ar[l] & & & & & -7
}
$$
\caption{ On the left, the complex of the mirror image knot $CFK^-(-T(3,4))$. On the right, the complex $\underline{CFK}^-(T(2,7))\otimes_{\F[U]} CFK^-(-T(3,4)).$ Note that it retains the same \textquotedblleft staircase" form, but some horizontal arrows get bent downward. }
\label{t27-t34}
\end{figure}
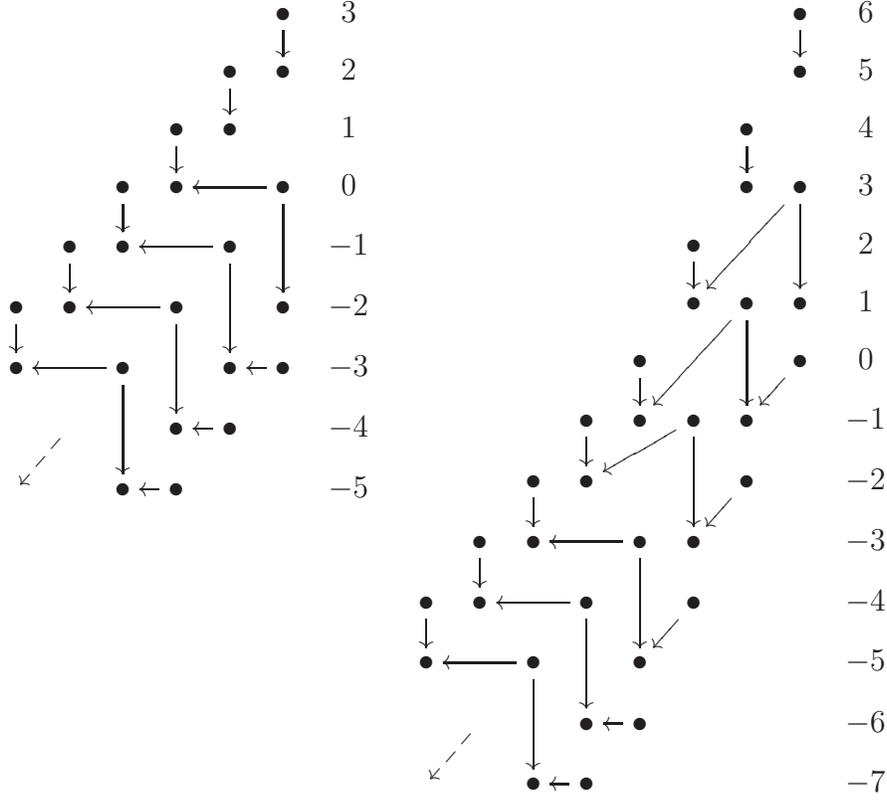

\begin{remark}
In particular, $C$ still has the properties in Remark \ref{upsidedown}, splitting into summands $C_{2i}$. Since the generator of homology of $C_0$ is the sum of {\it{all}} homogeneous elements with Maslov grading 0, its Alexander grading is the {\it{maximum}} of the Alexander gradings of all of these elements (see Equation \eqref{maxalex}). So, in this case, the $d_1$-invariant is zero if and only if {\it{all}} of the elements with Maslov grading zero have Alexander grading less than or equal to zero.
\end{remark}

We will now see how the Alexander grading on the $C_0$ summand is determined from the shape of the staircases; i.e., from the $\alpha_i$'s and $\beta_i$'s. Let us choose generators $\{ x_{-n}, \ldots, x_n\}$ for $CFK^-(K_1)$ as in Proposition \ref{lspacechar}. Then the generators for $\underline{CFK}^-(K_1)$ are $$x_{n-2i}, Ux_{n-2i}, \ldots ,U^{\alpha_{i+1} -1}x_{n-2i}$$ for every $0\leq i\leq n-1$, and $U^kx_{-n}$, for all $k\geq 0$. Further, as in the proof of Corollary \ref{lspacered}, we have that $$U( U^{\alpha_{i+1} -1}x_{n-2i}) = x_{n-2(i+1)} \ \ \ \textrm{for all} \ \ 0\leq i\leq n-1 $$ We should also point out that if $0\leq k <\alpha_{i+1}$, 

\begin{equation}
\label{filtlevel}
A(U^kx_{n-2i}) = \tau(K_1) - k  -\sum_{1\leq j\leq i}(\alpha_j  +\alpha_{n+1-j}),
\end{equation}
and \[M(U^kx_{n-2i}) = -2k-\sum_{1\leq j\leq i}2\alpha_j.\]

For $CFK^-(-K_2)$, we choose a basis $\{y_{-m},\ldots, y_m\}$, so that \[A(y_{-m+2k}) = \tau(-K_2) + \sum_{1\leq j\leq k}(\beta_{m+1-j}+\beta_j)\] and \[M(y_{-m+2k}) = \sum_{1\leq j\leq k}2\beta_j.\]

We now consider each generator in $C_0$ which has Maslov grading zero, and see what restrictions we get on the $\alpha_i$'s and $\beta_i$'s by assuming it has Alexander grading less than or equal to zero. The first generator we consider is $x_ny_{-m}$, and we have that
\begin{align*}
A(x_ny_{-m}) =&\  A(x_n)+A(y_{-m})\\
=& \ \tau(K_1)+\tau(-K_2)\\
=& \ \tau(K_1)-\tau(K_2).
\end{align*}
In order for this to be less than or equal to zero, we must have $\tau(K_1)\leq \tau(K_2)$. On the other hand, considering instead the knot $-K_1\# K_2$, the same argument says we must also have $\tau(K_2)\leq \tau(K_1)$, so $\tau(K_1)=\tau(K_2)$, and the Alexander polynomials of $K_1$ and $K_2$ must have equal degree.

The rest of the proof proceeds similarly. We next consider the generator $U^{\beta_1}x_ny_{-m+2}$. If $\alpha_1>\beta_1,$ then \[A(U^{\beta_1}x_ny_{-m+2}) = \tau(K_1)-\beta_1 +\tau(-K_2)+\beta_1+\beta_m>0,\] so, in order to have $d_1=0$, we must have $\alpha_1\leq \beta_1$. Again, considering $-K_1\# K_2$, we must also have $\beta_1\leq \alpha_1$, so $\alpha_1=\beta_1$.
Since $\alpha_1=\beta_1$, $U^{\beta_1}x_n = x_{n-2}$, so 
\begin{align*}
A(U^{\beta_1}x_ny_{-m+2}) =&\ A(x_{n-2}y_{-m+2})\\
=&\ \tau(K_1) - \alpha_1 - \alpha_n + \tau(-K_2)+\beta_1 +\beta_m\\
=&\ -\alpha_n +\beta_m.
\end{align*}
This means we must also have $\alpha_n\geq \beta_m$; and once again considering $-K_1\# K_2$, we see that in fact $\alpha_n=\beta_m$.

We have to this point shown that the first elements of the lists representing these two staircases agree, and also that the last elements agree. Taking this as our base case, we will now work our way inductively toward the middle.

To that end, assume that $\alpha_i=\beta_i$ and $\alpha_{n+1-i}=\beta_{m+1-i}$, for all $1\leq i\leq k$. 
Then consider the generator $$U^{\beta_1+\beta_2+\cdots +\beta_{k+1}}x_ny_{-m+2k+2} =  U^{\beta_{k+1}}x_{n-2k}y_{-m+2k+2}.$$
If $\alpha_{k+1}>\beta_{k+1}$, then 
\begin{align*}
A(U^{\beta_{k+1}}x_{n-2k}y_{-m+2k+2})=&\ A(U^{\beta_{k+1}}x_{n-2k})+A(y_{-m+2k+2})\\
=&\ \tau(K_1)-\beta_{k+1}-\sum_{1\leq j\leq k}(\alpha_j+\alpha_{n+1-j})\\
&\ -\tau(K_2) + \sum_{1\leq j\leq k+1} (\beta_j +\beta_{m+1-j})\\
=&\ \beta_{m-k}\\
>&\ 0,
\end{align*}
so it must be that $\alpha_{k+1}\leq \beta_{k+1}$; considering $-K_1\# K_2$ gives $\alpha_{k+1}=\beta_{k+1}$. 

Since $\alpha_{k+1}=\beta_{k+1}$, $U^{\beta_{k+1}}x_{n-2k}=x_{n-2k-2}$, so 

\begin{align*}
A(U^{\beta_{k+1}}x_{n-2k}y_{-m+2k+2}) =&\  A(x_{n-2k-2}y_{-m+2k+2})\\
=&\ \tau(K_1)-\sum_{1\leq j\leq k+1}(\alpha_j+\alpha_{n+1-j})\\
&\ -\tau(K_2) +  \sum_{1\leq j\leq k+1} (\beta_j +\beta_{m+1-j})\\
=&\ -\alpha_{n-k}+\beta_{m-k}
\end{align*}

This means $\alpha_{n-k}\geq \beta_{m-k}$, and as before, we see that $\alpha_{n-k}=\beta_{m-k}$, which completes the inductive step. {\it{A priori}}, $n$ may not be equal to $m$, but this induction can be continued for all $i$ until either $\alpha_i$ or $\beta_i$ does not exist. That is, until we exceed the minimum of $n$ and $m$. Assume, without loss of generality, that it is $n$. Upon reaching that point, we have $\alpha_i=\beta_i$, for $1\leq i\leq n$. But since the Alexander polynomials have equal degree, $$\sum_{1\leq i\leq n}\alpha_i = \sum_{1\leq i \leq m} \beta_i,$$ so $n$ and $m$ must be equal.
\end{proof}

\begin{figure}
$$\xymatrixcolsep{0.6 pc}\xymatrixrowsep{0.2 pc}
\xymatrix{
& & & & & & & \bu \ar[dd] & 6 \\
& & & & & & & & \\
& & & & & & & \bu & 5 \\
& & & & & & & & \\
& & & & & & \bu \ar[dd] & & 4 \\
& & & & & & & & \\
& & & & & & \bu & \bu \ar[lldddd] \ar[dddd] & 3 \\
& & & & & & & & \\
& & & & & \bu \ar[dd] & & & 2 \\
& & & & & & & & \\
& & & & & \bu & \tc{red}{\bullet} \ar@[red][lldddd] \ar@[red][dddd] & \bu & 1 \\
\ar@{--}[rrrrrrrr] & & & & & & & & \\
& & & & \tc{red}{\bu} \ar@[red][dd] & & & \tc{red}{\bu} \ar@[red][ldd]  & 0 \\
& & & & & & & & \\
& & & \bu \ar[dd] & \tc{red}{\bu} & \bu \ar[lldd] \ar[dddd] & \tc{red}{\bu}  & & -1 \\
& & & & & & & & \\
& & \bu \ar[dd]  & \bu & & & \bu \ar[ldd] & & -2 \\
& & & & & & & & \\
& \bu \ar[dd] & \bu & & \bu \ar[ll] \ar[dddd] & \bu & & & -3 \\
& & & & & & & & \\
\bu \ar[dd] & \bu & & \bu \ar[ll] \ar[dddd] & & \bu \ar[ldd] & & & -4 \\
& & & & & & & & \\
\bu & & \bu \ar[ll] \ar[dddd] & & \bu & & & & -5 \\
& & & & & & & & \\
& \da[ldd] & & \bu & \bu \ar[l] & & & & -6 \\
& & & & & & & & \\
& & \bu & \bu \ar[l] & & & & & -7
}$$
\caption{ The complex $\underline{CFK}^-(T(2,7))\otimes_{\F[U]} CFK^-(-T(3,4))$. Multiplication by $U$ takes staircase to staircase, but we suppress the dotted arrows to avoid obscuring the picture. Although $\tau\big(T(2,7)\#-T(3,4)\big)=0$, the summand $C_0$ (shaded red) has its generator of homology with Alexander grading 1 (above the dashed line), so $d_1\big(T(2,7)\#-T(3,4)\big)\neq 0$ (in fact, $d_1=-2$). }
\label{d>0}
\end{figure}
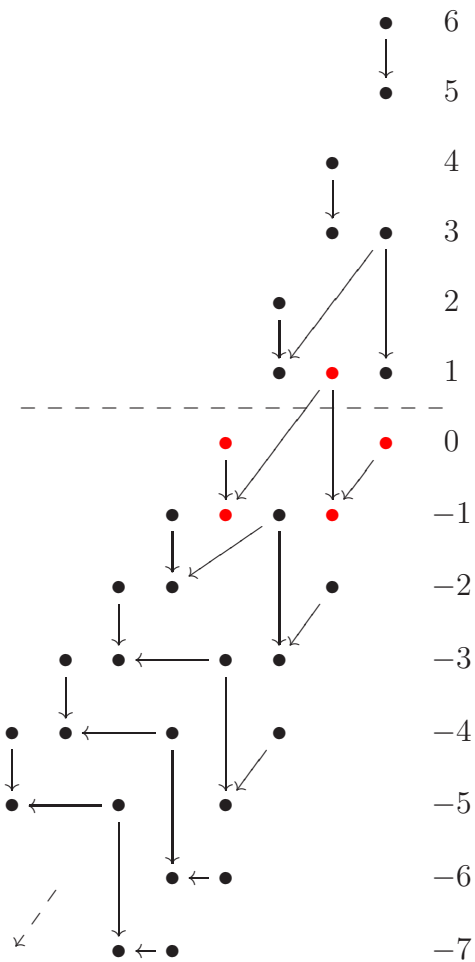

Figure \ref{t27-t34} shows the example of the sum $T(2,7)\# -T(3,4)$. In Figure \ref{d>0}, it is exhibited that \[d_1(T(2,7)\# -T(3,4)) =-2,\] although \[\tau(T(2,7)\# -T(3,4)) =0.\] This is an instance of a general fact which follows from Theorem \ref{alexanderconcordance}.

\begin{corollary}
If $K_1$ and $K_2$ are two $L$-space knots whose Alexander polynomials are distinct but have the same degree, then $$\tau(K_1\# -K_2) = \tau (-K_1\# K_2) =0,$$ but either $$d_1(K_1\# -K_2)\neq 0 \ \ \ \textrm{ or } \ \ \ d_1(-K_1\# K_2)\neq 0.$$ In particular, $d_1$ gives a stronger four-genus bound than $\tau$ for $K_1\# -K_2$ and its mirror.
\label{lspacecor}
\end{corollary}

\noindent {\bf{Example 1.}} The example illustrated in Figure \ref{d>0} can be generalized to the knots \[K_p:= T\left( 2,p(p-1)+1\right)\# -T(p,p+1).\] By examining the Alexander polynomials of torus knots, it can be seen that the lengths of the staircase for $T(p,p+1)$ are $\{ 1,2,\cdots, p-1 \}$, whereas the staircase for $T(2,q)$ has $\frac{q-1}{2}$ steps, all of length 1.

From this it can be seen (see Figure \ref{t231-t67} for an example) that the generator of homology of the complex $C_0$ has Alexander grading
\begin{align}
\begin{split}
A =&\ \max_k \sum_{i=1}^k (p-i)-i \\
=&\ \sum_{i=1}^{\lfloor \frac{p}{2}\rfloor } p-2i\\
=&\ \left\lfloor \frac{p}{2}\right\rfloor \left( p-\left\lfloor \frac{p}{2}\right\rfloor-1\right).
\end{split}
\label{A0}
\end{align}
In general, showing that the Alexander grading of this generator is positive only shows that $d_1\leq 0$, but in this case we can get an explicit value with relative ease. Roughly speaking, this is because the map $U$ on $\underline{CFK}^-(T(2,p(p-1)+1))$ decreases the Alexander grading by 2 (at least on elements with high enough Maslov grading), and of course also decreases the Maslov grading by 2. So, in fact, the grading in \eqref{A0} is exactly $-d_1$. That is, for any $p>1$,
\begin{equation}
\tau(K_p)=0, \ \ \ \ \  \text{but} \ \ \ \ \ 
d_1(K_p) =
\begin{cases}
-\frac{p^2-2p}{4} &  p\  \text{ even},\\
-\left( \frac{p-1}{2}\right)^2 & p \equiv 1 \mod 4,\\
-\left( \frac{p-1}{2}\right)^2 -1 & p \equiv 3 \mod 4.
\end{cases}
\end{equation}
It should be pointed out that while $d_1$ has more to say than $\tau$ for these knots, the knot signature $\sigma$ gives an even better topological four-genus bound (at least for $p>5$).

\begin{figure}
$$\xymatrixcolsep{0.3 pc} \xymatrixrowsep{0.3 pc}
\xymatrix{
 & & &  & & & & & & & & & & \bu \ar[ddd] \ar[ddddllll] & & & \bu \ar[dddd] \ar[dddlll] & & & & 6 \\
 & & &  & & & & & & & & & & & & & & & & &  \\
 & & &  & & & & & & \bu \ar[dd] \ar[dddddlllll] & & & & & & & & & \bu \ar[ddddd] \ar[ddll] & & 4 \\
 & & &  & & & & & & & & & & \bu & & & & & & &  \\
 & & &  & & & & & & \bu & & & & & & & \bu & & & & 2 \\
 \ar@{--}[rrrrrrrrrrrrrrrrrrrr] & & & &  & & & & & & & & & & & & & & & &  \\
 & & &  & \bu \ar[d] & & & & & & \bu \ar[ddd] \ar[dldldldl] & &  & \bu \ar[dddd] \ar[dldldl] &  & & & & & \bu \ar[dl] & 0 \\
 & & &  & \bu & & & & & & & & & & & & & & \bu & &  \\
 & & & & & & \bu \ar[dd] \ar[ddlllll] & & & & & & & & & \bu \ar[ddddd] \ar[ddll] & & & & &  -2 \\
 & \bu \ar[d] & & & & & & & & & \bu & & & & & & &  & & &  \\
 & \bu & & & & & \bu & & & & & & & \bu & & & &  & & & -4 \\
 & & & & & & & & & & & & & & & & &  & & &  \\
 & & & & & & & & & & & & & & & & \bu \ar[dl] & & & &  -6 \\
 & & & & & & & & & & & & & & & \bu & &  & & &  \\
 }
$$
\caption{A portion of the complex $\underline{CFK}^-(T(2,31))\otimes_{\F[U]}CFK^-(-T(6,7))$. The upper summand shown is $C_0$, and its generator of homology has Alexander grading 6. The generators of homology for $C_{-2}$ and $C_{-4}$ (not shown) are also above the dashed line. The generator of homology of the summand $C_{-6}$, the lower summand shown, has the maximal Maslov grading of any generator below the dashed line, so $d_1(K_6)=-6$. }
\label{t231-t67}
\end{figure}
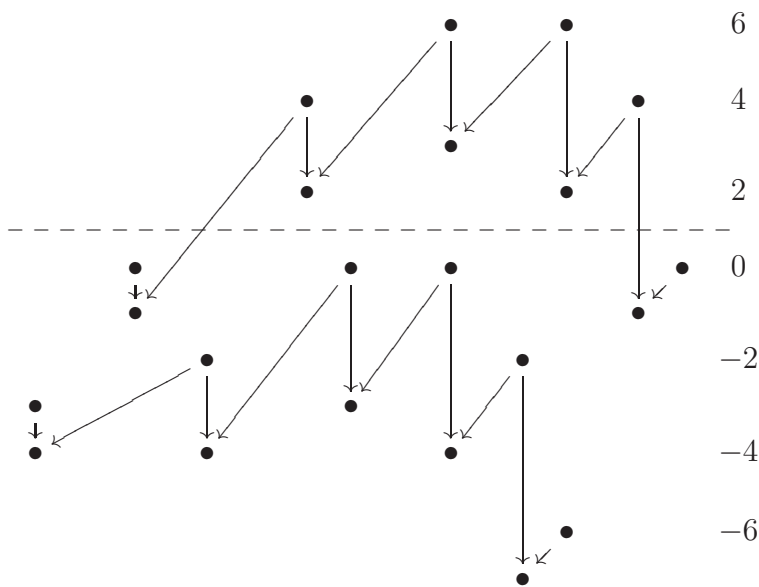

\noindent {\bf{Example 2.}} Even among sums of torus knots, however, we can find examples for which $\tau$, $\sigma$ and Rasmussen's $s$ invariant defined using Khovanov homology \cite{RasmussenSlice} are all equal to zero, and $|d_1|$ is arbitrarily large (for all knots discussed in this paper, $s=2\tau$). Define \[J_p^+:=T(2,8p+1)\# T(4p,4p+1),\] and then let \[J_p:= J_p^+ \# -T(4p+1,4p+2).\] A direct computation using \eqref{tautorus} and, for example, \cite[Theorem 5.2]{GLM} shows that \[ \tau(J_p)=\sigma (J_p) =0 \ \ \ \text{for all } p>0.\]

The staircases of the individual torus knot summands here have the type described in Example 1. The sum of two $L$-space knots, as we have seen, is not an $L$-space knot. However, its reduced complex has an acyclic subcomplex which is $U$-torsion, and the corresponding quotient complex is isomorphic to $\F[U]$. Since $U$-torsion elements are not relevant to the computation of $d$-invariants, this means we can treat the sum of staircases as a staircase, if it is only the $d$-invariants we are interested in (see \cite[Section~5]{BL} and \cite[Section~2.4]{BLSemi}, where Borodzik and Livingston discuss the gap functions of connected sums of algebraic knots, for an alternate point of view). That is to say, this quotient complex is filtered isomorphic to the reduced complex corresponding to some staircase, which we may call the \textquotedblleft representative staircase".

 If one of the summands is $T(2,n)$, the resulting representative staircase can be obtained relatively simply. In the case at hand, the representative staircase for $J_p^+$ is given by \[ \{\overbrace{1,\cdots,1}^{2p^2+5p},3,\overbrace{1,\cdots,1}^{2p-2},5,\overbrace{1,\cdots,1}^{2p-3},\cdots, 4p-5,1,1,4p-3,1,4p-1 \}.   \]
An example of this staircase is shown in Figure \ref{j2+}, for the case $p=2$. Recall that, by the symmetry of the Alexander polynomials of the summands of $J_p^+$, this is also the list of vertical lengths, from bottom to top. With this, and also knowing the upside-down staircase shape of $CFK^-(T(4p+1,4p+2))$, we can compute the Alexander gradings as we did above. This allows us to see that, for all $p>0$,
\begin{equation}
d_1(J_p) = -2p. 
\end{equation}
We suppress the explicit computations here, but instead show the $p=2$ case in Figure \ref{j2}.

\begin{figure}
$$\xymatrixcolsep{0.015 pc} \xymatrixrowsep{0.015 pc}
\xymatrix@=0pt{
&&&&&  &&&&&  &&&&&  &&&&&  &&&&& &&& \\
&&&&&  &&&&&  &&&&&  &&&&&  &&&&& &&& \\
&&  &&&  &&&&&  &&&&&  &&&&&  &&&&& &&& \\
&&&&&  &&&&&  &&&&&  &&&&&  &&&&& &&&  \\
&&  &&&  &&&&&  &&&&&  &&&&&  &&&&& &&& \\
&&&&&  &&&&&  &&&&&  &&&&&  &&&&& &&&  \\
&&&&&  &&&&&  &&&&&  &&&&&  &&&&& &&& \\
&&&&&  &&&&&  &&&&&  &&&&&  &&&&& &&& \\
&&&&&  &&&&&  &&&&&  &&&&&  &&&&& &&& \\
 \lar'[r]+0 '[rddddddd]+0  '[rrrddddddd]+0   '[rrrddddddddddddd]+0   '[rrrrrrddddddddddddd]+0  '[rrrrrrdddddddddddddddddd]+0  '[rrrrrrrrrrdddddddddddddddddd]+0  '[rrrrrrrrrrdddddddddddddddddddddd]+0  '[rrrrrrrrrrrrrrrdddddddddddddddddddddd]+0   '[rrrrrrrrrrrrrrrddddddddddddddddddddddddd]+0   '[rrrrrrrrrrrrrrrrrrrrrddddddddddddddddddddddddd]+0     '[rrrrrrrrrrrrrrrrrrrrrddddddddddddddddddddddddddd]+0  '[rrrrrrrrrrrrrrrrrrrrrrrrrrrrddddddddddddddddddddddddddd]+0                                                                        '[rrrrrrrrrrrrrrrrrrrrrrrrrrrrdddddddddddddddddddddddddddd]+0                                                                      &&&&&  &&&&&  &&&&&  &&&&&  &&&&& &&& \\
&&&&&  &&&&&  &&&&&  &&&&&  &&&&& &&& \\
&&  &&&  &&&&&  &&&&&  &&&&&  &&&&& &&& \\
&&&&&  &&&&&  &&&&&  &&&&&  &&&&& &&&  \\
&&&&&  &&&&&  &&&&&  &&&&&  &&&&& &&& \\
&&&&&  &&&&&  &&&&&  &&&&&  &&&&& &&& \\
&&&&&  &&&&&  &&&&&  &&&&&  &&&&& &&& \\
&&&  &&  &&&&&  &&&&&  &&&&&  &&&&& &&& \\
&&&&&  &&&&&  &&&&&  &&&&&  &&&&& &&&  \\
&&&&&  &&&&&  &&&&&  &&&&&  &&&&& &&& \\
&&&&&  &&&&&  &&&&&  &&&&&  &&&&& &&& \\
&&&&&  &&&&&  &&&&&  &&&&&  &&&&& &&& \\
&&&&&  &&&&&  &&&&&  &&&&&  &&&&& &&& \\
&&&&&  & &&&&  &&&&&  &&&&&  &&&&& &&&  \\
&&&&&  &&&&&  &&&&&  &&&&&  &&&&& &&& \\
&&&&&  &&&&&  &&&&&  &&&&&  &&&&& &&& \\
&&&&&  &&&&&  &&&&&  &&&&&  &&&&& &&& \\
&&&&&  &&&&&  &&&&&  &&&&&  &&&&& &&& \\
&&&&&  &&&&&  &&&&&  &&&&&  &&&&& &&&  \\
&&&&&  &&&&&  &&&&&  &&&&&  &&&&& &&& \\
&&&&&  &&&&&  &&&&&  &&&&&  &&&&& &&& \\
&&&&&  &&&&&  &&&&&  &&&&&  &&&&& &&& \\
&&&&&  &&&&&  &&&&&  &&&&&  &&&&& &&& \\
&&&&&  &&&&&  &&&&&  &&&&&  &&&&& &&&  \\
&&&&&  &&&&&  &&&&&  &&&&&  &&&&& &&& \\
&&&&&  &&&&&  &&&&&  &&&&&  &  &&&& &&& \\
&&&&&  &&&&&  &&&&&  &&&&&  &&&&& &&& \\
&&&&&  &&&&&  &&&&&  &&&&&  &&&&& &&&  \\ 
&&&&&  &&&&&  &&&&&  &&&&&  &&&&& &&& 
}  \!\!\!\!\!\!\!\!\!\!\!\!\!\!\!\!\!\!\!\!\!\!\!\!\!\!\!\!\!\!\!\!\!\!\!\!\!\!\!\!\!\!\!\!\!\!\!\!\!\!\!\!\!\!\!
\xymatrix{
\lar'[0,1]+0 '[7,1]+0 '[7,2]+0 '[8,2]+0 '[8,3]+0 '[13,3]+0 '[13,4]+0 '[14,4]+0 '[14,5]+0 '[15,5]+0 '[15,6]+0  '[18,6]+0  '[18,7]+0  '[19,7]+0 '[19,8]+0 '[20,8]+0 '[20,9]+0 '[21,9]+0 '[21,10]+0 '[22,10]+0 '[22,11]+0  '[23,11]+0  '[23,12]+0  '[24,12]+0  '[24,13]+0 '[25,13]+0  '[25,14]+0  '[26,14]+0  '[26,15]+0  '[27,15]+0  '[27,16]+0 '[28,16]+0  '[28,17]+0  '[29,17]+0  '[29,18]+0 '[30,18]+0  '[30,21]+0  '[31,21]+0  '[31,22]+0  '[32,22]+0  '[32,23]+0  '[33,23]+0  '[33,28]+0  '[34,28]+0  '[34,29]+0  '[35,29]+0  '[35,36]+0  '[36,36]+0
&&&&&& &&&&&& &&&&&& &&&&&& &&&&&& &&&&&&  \\
&&&&&& &&&&&& &&&&&& &&&&&& &&&&&& &&&&&&  \\
&&&&&& &&&&&& &&&&&& &&&&&& &&&&&& &&&&&&  \\
&&&&&& &&&&&& &&&&&& &&&&&& &&&&&& &&&&&&  \\
&&&&&& &&&&&& &&&&&& &&&&&& &&&&&& &&&&&&  \\
&&&&&& &&&&&& &&&&&& &&&&&& &&&&&& &&&&&&  \\
&&&&&& &&&&&& &&&&&& &&&&&& &&&&&& &&&&&&  \\
&&&&&& &&&&&& &&&&&& &&&&&& &&&&&& &&&&&&  \\
&&&&&& &&&&&& &&&&&& &&&&&& &&&&&& &&&&&&  \\
&&&&&& &&&&&& &&&&&& &&&&&& &&&&&& &&&&&&  \\
&&&&&& &&&&&& &&&&&& &&&&&& &&&&&& &&&&&&  \\
&&&&&& &&&&&& &&&&&& &&&&&& &&&&&& &&&&&&  \\
&&&&&& &&&&&& &&&&&& &&&&&& &&&&&& &&&&&&  \\
&&&&&& &&&&&& &&&&&& &&&&&& &&&&&& &&&&&&  \\
&&&&&& &&&&&& &&&&&& &&&&&& &&&&&& &&&&&&  \\
&&&&&& &&&&&& &&&&&& &&&&&& &&&&&& &&&&&&  \\
&&&&&& &&&&&& &&&&&& &&&&&& &&&&&& &&&&&&  \\
&&&&&& &&&&&& &&&&&& &&&&&& &&&&&& &&&&&&  \\
&&&&&& &&&&&& &&&&&& &&&&&& &&&&&& &&&&&&  \\
&&&&&& &&&&&& &&&&&& &&&&&& &&&&&& &&&&&&  \\
&&&&&& &&&&&& &&&&&& &&&&&& &&&&&& &&&&&&  \\
&&&&&& &&&&&& &&&&&& &&&&&& &&&&&& &&&&&&  \\
&&&&&& &&&&&& &&&&&& &&&&&& &&&&&& &&&&&&  \\
&&&&&& &&&&&& &&&&&& &&&&&& &&&&&& &&&&&&  \\
&&&&&& &&&&&& &&&&&& &&&&&& &&&&&& &&&&&&  \\
&&&&&& &&&&&& &&&&&& &&&&&& &&&&&& &&&&&&  \\
&&&&&& &&&&&& &&&&&& &&&&&& &&&&&& &&&&&&  \\
&&&&&& &&&&&& &&&&&& &&&&&& &&&&&& &&&&&&  \\
&&&&&& &&&&&& &&&&&& &&&&&& &&&&&& &&&&&&  \\
&&&&&& &&&&&& &&&&&& &&&&&& &&&&&& &&&&&&  \\
&&&&&& &&&&&& &&&&&& &&&&&& &&&&&& &&&&&&  \\
&&&&&& &&&&&& &&&&&& &&&&&& &&&&&& &&&&&&  \\
&&&&&& &&&&&& &&&&&& &&&&&& &&&&&& &&&&&&  \\
&&&&&& &&&&&& &&&&&& &&&&&& &&&&&& &&&&&&  \\
&&&&&& &&&&&& &&&&&& &&&&&& &&&&&& &&&&&&  \\
&&&&&& &&&&&& &&&&&& &&&&&& &&&&&& &&&&&&  \\
&&&&&& &&&&&& &&&&&& &&&&&& &&&&&& &&&&&&
}
$$
\caption{On the left is the staircase for the torus knot $T(8,9)$ (the horizontal lengths range from 1 to 7, in order). On the right is the representative staircase corresponding to the knot $T(2,17)\# T(8,9)$, which we have called $J_2^+$. This representative staircase contains all of the generators which are relevant for computing $d$-invariants.}
\label{j2+}
\end{figure}
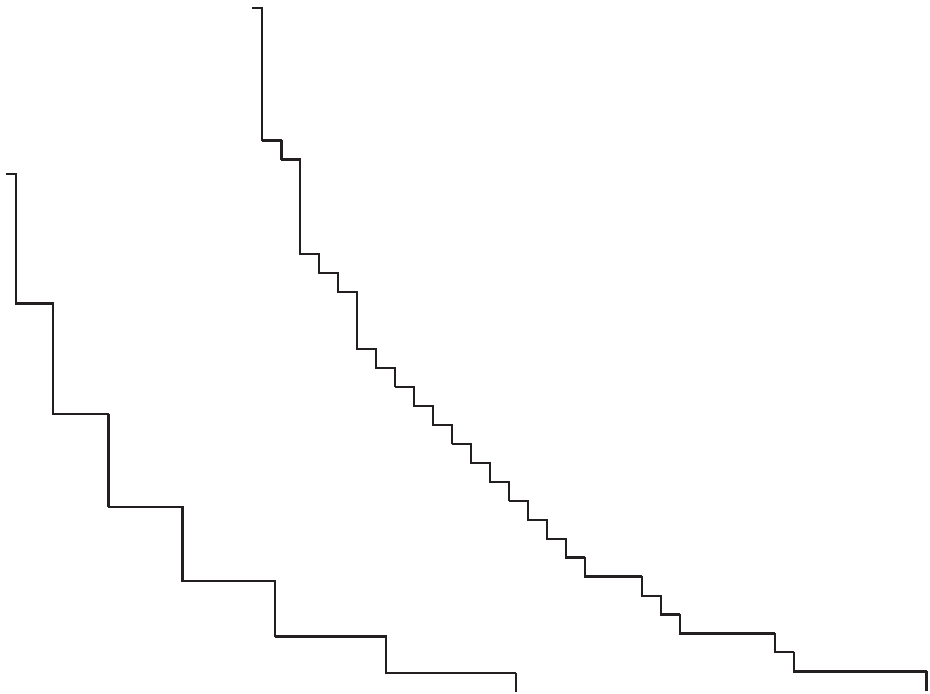

\begin{figure}
$$\xymatrixcolsep{0.02 pc}  \xymatrixrowsep{0.3 pc} \xymatrix{
&&&&&& &&&&&& &&&&&& &&&&&& &&&&&& &&&&&& &&& 4 \\
&&&&&& &&&&&& &&&&&& &&&&&& &&&&&& &&&&&& &&& 3 \\
&&&&&& &&&&&& &&&&&& &&&&&& &&&&&& &&&&&& &&& 2 \\
&&&&&& &&&&&& &&&&&& &&&&&& &&&&&& &&&&&& &&& 1 \\
\ar@{--}[0,39] &&&&&& &&&&&& &&&&&& &&&&&& &&&&&& &&&&&& &&& \\
&& \lar'[1,0]+0 '[-2,8]+0 '[1,8]+0 '[-3,15]+0 '[1,15]+0 '[-4,21]+0 '[1,21]+0 '[-5,26]+0*{\times} '[1,26]+0 '[-4,30]+0 '[3,30]+0 '[-3,33]+0 '[5,33]+0 '[-2,35]+0 '[7,35]+0 '[0,36]+0     &&&& &&&&&& &&&&&& &&&&&& &&&&&& &&&&&& &&& 0 \\
&   \lar'[1,0]+0 '[0,8]+0 '[2,8]+0 '[-1,15]+0 '[2,15]+0 '[-3,21]+0 '[2,21]+0 '[-4,26]+0 '[2,26]+0 '[-3,30]+0 '[4,30]+0 '[-1,33]+0 '[6,33]+0 '[0,35]+0 '[8,35]+0 '[7,36]+0  &&&&& &&&&&& &&&&&& &&&&&& &&&&&& &&&&&& &&& -1 \\
\lar'[1,0]+0 '[0,8]+0 '[2,8]+0 '[0,15]+0 '[3,15]+0 '[-1,21]+0 '[3,21]+0 '[-2,26]+0 '[3,26]+0 '[-1,30]+0 '[5,30]+0 '[0,33]+0 '[7,33]+0 '[5,35]+0 '[13,35]+0 '[8,36]+0   &&&&&& &&&&&& &&&&&& &&&&&& &&&&&& &&&&&& &&& -2 \\
&&&&&& &&&&&& &&&&&& &&&&&& &&&&&& &&&&&& &&& \\
&&&&&& &&&&&& &&&&&& &&&&&& &&&&&& &&&&&& &&& \\
&&&&&& &&&&&& &&&&&& &&&&&& &&&&&& &&&&&& &&& \\
&&&&&& &&&&&& &&&&&& &&&&&& &&&&&& &&&&&& &&& \\
&&&&&& &&&&&& &&&&&& &&&&&& &&&&&& &&&&&& &&& \\
&&&&&& &&&&&& &&&&&& &&&&&& &&&&&& &&&&&& &&& \\
&&&&&& &&&&&& &&&&&& &&&&&& &&&&&& &&&&&& &&& \\
&&&&&& &&&&&& &&&&&& &&&&&& &&&&&& &&&&&& &&& \\
&&&&&& &&&&&& &&&&&& &&&&&& &&&&&& &&&&&& &&& \\
&&&&&& &&&&&& &&&&&& &&&&&& &&&&&& &&&&&& &&& \\
&&&&&& &&&&&& &&&&&& &&&&&& &&&&&& &&&&&& &&& \\
&&&&&& &&&&&& &&&&&& &&&&&& &&&&&& &&&&&& &&& \\
&&&&&& &&&&&& &&&&&& &&&&&& &&&&&& &&&&&& &&&
}
$$
\caption{The portion of the complex $\underline{CFK}^-(J_p)$ which is relevant for computing $d_1(J_p)$, in the case where $p=2$. The uppermost summand is $C_0$. From right to left, its generators start with an Alexander grading of 0, and increase by 1 until reaching an Alexander grading of $2p$ (the generator marked with an $\times$). Multiplying by $U$ takes staircase to staircase, and notice that the Alexander grading of the $\times$ generator decreases by 2 each time. This is the basic idea behind showing that $d_1(J_p)=-2p$.}
\label{j2}
\end{figure}
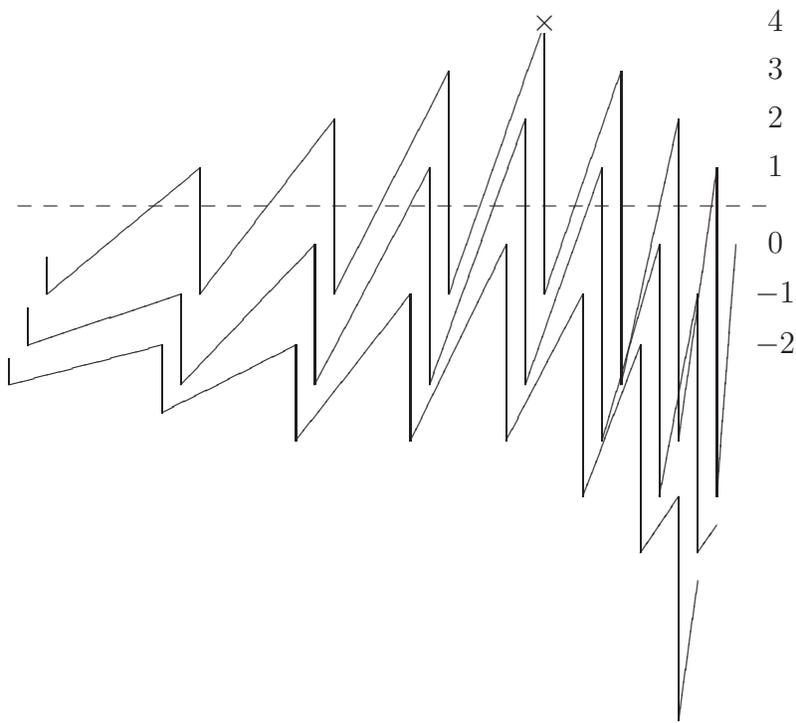

There is, more generally, a family of knots whose knot Floer complexes are the direct sum of a staircase and an acyclic complex. In addition to $L$-space knots, Petkova \cite[Lemma 7]{Petkova} showed that \textquotedblleft Floer homologically thin" knots -- which include alternating and quasi-alternating knots, as well as a family of hyperbolic knots found in \cite{GreeneWatson} -- are in this family (the staircase for a Floer homologically thin knot is the same as that of some $(2,n)$-torus knot).  Since $d$-invariants are defined in terms of non-torsion generators of homology, the acyclic summands have no effect on $d_1$; it is determined solely by the staircase summand. Therefore, if $K_1$ and $K_2$ are knots in this family, then $d_1(K_1\# -K_2)$ can be computed as it was in the above examples for torus knots. In particular, if the staircases of $K_1$ and $K_2$ have different shapes, these knots are not concordant. 

As an interesting further application of these ideas, one could investigate the linear independence of a family of knots in the smooth concordance group. As an example, if we let $T(r,s)_{p,q}$ denote the $(p,q)$-cable of $T(r,s)$, the knots \[K_1 = 
T(2,3)_{2,3}\# T(2,5)\ \text{ and}\ \ K_2= T(2,3)_{2,5}\# T(2,3) \] can be shown to be linearly independent using $d_1$, although \( \Delta_{K_1}(T)=\Delta_{K_2}(T)\) (showing that Corollary \ref{lspacecor} does not extend to sums of $L$-space knots). In contrast, the knots \[ T(2,3)_{2,13}\# T(2,15) \ \ \text{and}\ \  T(2,3)_{2,15}\# T(2,13) \] cannot be distinguished in the concordance group using $d_1$ (i.e.; $d_1$ cannot obstruct the sliceness of the \textquotedblleft Livingston-Melvin" knot \cite{LivingstonMelvin}). So, while Theorem \ref{alexanderconcordance} settles the question of when two $L$-space knots are concordant, it would be interesting to understand which families of $L$-space knots can be shown to be independent using the $d_1$ invariant.

\bibliography{bibCFK2}{}
\bibliographystyle{amsplain}

\end{document}